\documentclass[12pt,reqno,a4paper]{amsart}    

\usepackage[totalwidth=450pt, totalheight=717pt]{geometry}

\usepackage{amsmath, amssymb}
\usepackage{amsthm}

\usepackage{accents}     

\usepackage{graphicx}                
\newcommand{\color}[2][{}]{}         
\usepackage{caption}
\usepackage{bbm}         

\usepackage{ifthen}       

\usepackage{mathrsfs}    
\renewcommand\mathcal\mathscr  

\numberwithin{equation}{section}

\newcounter{myenumi}
\newenvironment{myenumerate}[1]{
\begin{list}{\indent(\themyenumi) }
  {\renewcommand{\themyenumi}{#1{myenumi}}
    \usecounter{myenumi}
    \setlength{\topsep}{0em}
    \setlength{\itemsep}{0em}
    \setlength{\leftmargin}{0em}
    \setlength{\labelwidth}{0em}
    \setlength{\labelsep}{0em}
  }
  }
  {
  \end{list}
  }
\newcommand{\itemref}[1]{\eqref{#1}}

\theoremstyle{plain}            
\newtheorem{theorem}{Theorem}[section]

\newtheorem{proposition}[theorem]{Proposition}
\newtheorem{lemma}[theorem]{Lemma}
\newtheorem{corollary}[theorem]{Corollary}

\theoremstyle{definition}       
\newtheorem{definition}[theorem]{Definition}

\newtheorem{remark}[theorem]{Remark}
\newtheorem*{remark*}{Remark}
\newtheorem{remarks}[theorem]{Remarks}

\newcommand{\Sec}[1]{Section~\ref{sec:#1}}

\newcommand{\App}[1]{Appendix~\ref{app:#1}}

\newcommand{\Fig}[1]{Figure~\ref{fig:#1}}

\newcommand{\Thm}[1]{Theorem~\ref{thm:#1}}

\newcommand{\Lem}[1]{Lemma~\ref{lem:#1}}

\newcommand{\Lemenum}[2]{Lemma~\ref{lem:#1}~(\ref{#2})}

\newcommand{\Cor}[1]{Corollary~\ref{cor:#1}}

\newcommand{\Prp}[1]{Proposition~\ref{prp:#1}}

\newcommand{\Rem}[1]{Remark~\ref{rem:#1}}

\newcommand{\Def}[1]{Definition~\ref{def:#1}}

\newcommand{\abs}[2][{}]{\lvert{#2}\rvert_{{#1}}}    
\newcommand{\abssqr}[2][{}]{\lvert{#2}\rvert^2_{#1}} 
\newcommand{\bigabs}[2][{}]{\bigl\lvert{#2}\bigr\rvert_{#1}}     
\newcommand{\Bigabs}[2][{}]{\Bigl\lvert{#2}\Bigr\rvert_{#1}}     

\newcommand{\normsymb}{|\!|}

\newcommand{\bignormsymb}[1]{#1|\!#1|}

\newcommand{\norm}[2][{}]{\normsymb{#2}\normsymb_{{#1}}}    
\newcommand{\normsqr}[2][{}]{\normsymb{#2}\normsymb^2_{#1}} 
\newcommand{\bignorm}[2][{}]{\bignormsymb{\bigl}{#2}\bignormsymb{\bigr}_{#1}}

\newcommand{\scapro}[3][{}]{\langle{#2},{#3}\rangle_{#1}}  
\newcommand{\bigscapro}[3][{}]{\bigl\langle{#2},{#3}\bigr\rangle_{#1}}

\newcommand{\set}[2]{\{ \, #1 \, ; \, #2 \, \} }      
\newcommand{\bigset}[2]{\bigl\{ \, #1 \, ; \, #2 \, \bigr\} }

\newcommand{\map}[3]{ #1 \colon #2 \longrightarrow #3}    

\newcommand{\clo}[1]{\overline{{#1}}} 

\newcommand{\conj}[1]{\overline {{#1}}}  

\newcommand{\dd}    {\, \mathrm d}    
\DeclareMathOperator{\dist}   {dist}
\DeclareMathOperator{\dom}    {Dom}
\DeclareMathOperator{\ran}    {Ran}
\DeclareMathOperator{\inj}    {inj}  
\DeclareMathOperator{\Ric}    {Ric}

\DeclareMathOperator{\supp}   {supp}

\DeclareMathOperator{\dvol}    {d\, vol}

\newcommand{\strong}          {\mathrm s}    
\DeclareMathOperator{\slim}   {\strong\text{-}lim}  

\newcommand{\specsymb} {\sigma} 

\newcommand{\spec}[2][{}]   {\specsymb_{\mathrm{#1}}(#2)}

\newcommand{\eps}{\varepsilon} 
\renewcommand{\phi}{\varphi}   
\renewcommand{\rho}{\varrho}   
\renewcommand{\theta}{\vartheta}

\newcommand{\R}{\mathbb{R}} 
\newcommand{\C}{\mathbb{C}} 
\newcommand{\N}{\mathbb{N}} 

\newcommand{\e}{\mathrm e}  
\newcommand{\im}{\mathrm i} 

\newcommand{\wt}{\widetilde}           
\renewcommand{\mid}{\hskip.5ex ; \hskip.5ex}

\newcommand{\Sobsymb} {\mathsf H}      
\newcommand{\Sobnsymb} {\ring{\mathsf H}}   
\newcommand{\SobWsymb}{\mathsf W}      

\newcommand{\Contsymb} {\mathsf C}     
\newcommand{\Lsymb}    {\mathsf L}     

\newcommand{\Sobspace}[1][1]{\Sobsymb^{#1}} 
\newcommand{\Sobnspace}[1][1]{\Sobnsymb^{#1}} 

\newcommand{\SobWspace}[2][p]{\SobWsymb_{#1}^{#2}}  
  
\newcommand{\Contspace}[1][{}]{\Contsymb^{#1}}     
\newcommand{\Cispace}{\Contsymb^\infty}     
\newcommand{\Lpspace}[1][p]    {\Lsymb_{#1}}     
\newcommand{\Lsqrspace}    {\Lpspace[2]}     

\newcommand{\Ci} [2][{}]{\Cispace_{#1} ({#2})}
\newcommand{\Cci}[1]{\Ci[\mathrm c]{#1}}
\newcommand{\Cont}[2][{}]{\Contspace[#1]({#2})}

\newcommand{\Schwartz}[1]{\mathcal S(#1)}

\newcommand{\Lsqr}[2][{}]{\Lsqrspace^{#1}({#2})} 
\newcommand{\Lsqrloc}[2][{}]{\Lpspace [2,\mathrm{loc}]^{#1}({#2})}

\newcommand{\Linfty}[2][{}]{\Lpspace [\infty] ^{#1}({#2})} 

\newcommand{\Sob}[2][1]{\Sobspace [#1]({#2})}         
\newcommand{\Sobn}[2][1]{\Sobnspace [#1]({#2})}  

\newcommand{\Sobx}[3][1]{\Sobspace [#1]_{{#2}}({#3})} 
\newcommand{\SobW}[3][p]{\SobWspace[#1]{#2}(#3)} 

\newcommand{\Sobloc}[2][1]{\Sobx[#1]{\mathrm{loc}}{#2}}

\newcommand{\FF}{\mathcal F}
\newcommand{\HH}{\mathcal H}
\newcommand{\MM}{\mathcal M}

\newcommand{\XX}{\mathcal X}

\newcommand{\Met}{\mathsf{Met}}  

\newcommand {\upp}{\mathrm{u}}    
\newcommand {\low}{\ell}          

\newcommand {\dec}{\mathrm{dec}}
\newcommand {\ac} {\mathrm {ac}}

\newcommand {\ext} {\mathrm {ext}}
\newcommand {\interior} {\mathrm {int}}

\newcommand{\quadtext}[1]{\quad\text{#1}\quad}

\newcommand{\slimpm}{\slim_{t \to \pm \infty}}

\newcommand{\calC}{C} 

\newcommand{\Qp}{q_+} 
\newcommand{\Qm}{q_-} 
\newcommand{\Qpm}{q_\pm} 

\newcommand{\Pt}{p}  

\newcommand{\gEucl}{E}  
\renewcommand{\gEucl}{{g_{\mathrm{E}}}}  
\begin{document}

\title[Open Scattering Channels for a Two-Sheeted Covering]%
{On Open Scattering Channels for a Branched Covering of the Euclidean
  Plane.}

\author{Rainer Hempel}
\address{Institute for Computational Mathematics, TU Braunschweig,
  Universit\"atsplatz 2, 38106 Braunschweig, Germany}
\email{r.hempel@tu-bs.de}

\author{Olaf Post}
\address{Mathematik, Fachbereich 4, Universit\"at Trier, 54286 Trier, Germany}
\email{olaf.post@uni-trier.de}

\date{\today \quad  \emph{File:} \texttt{\jobname.tex}}


\begin{abstract}
   We study the interaction of two scattering channels for a simple
   geometric model consisting in a double covering of the plane with 
   two branch points, equipped with the Euclidean metric. 
  We show that the scattering channels are open in the sense of~\cite{hpw:14} 
 and that this property is stable under suitable perturbations of the metric.
\end{abstract}

\maketitle


%
\section{Introduction}
\label{sec:intro}
%
%
Let $M$ denote a branched covering of the plane, obtained by glueing
two copies of $\R^2$ along a straight-line cut between the points $\Qm
= (-1,0)$ and $\Qp = (+1,0)$, where the northern edge of the upper
copy of $\R^2$ is joined to the southern edge of the lower copy, and
vice-versa (see \Fig{1}). The branch points $\Qpm$ do \emph{not}
belong to $M$.
The manifold $M$ is a real version of the complex Riemann surface
associated with the function $\sqrt{z^2 - 1}$.  With the Euclidean
metric $\gEucl$ of $\R^2$, we obtain a smooth, connected Riemannian
manifold $\MM = (M,\gEucl)$ with curvature zero; note, however, that
$\MM$ is not complete. In the second part of the paper we will
consider Riemmannian metrics $g$ on $M$ which are close to $\gEucl$ in
a suitable sense so that the perturbational results of~\cite{hpw:14}
can be applied.

We let $H$ denote the Laplacian of $\MM$, a self-adjoint operator
acting in the Hilbert space $\HH = \Lsqr \MM$.  For a metric $g$ on
$M$, different from the Euclidean metric, we denote the associated
Laplacian by $H_g$.  It is the aim of this paper to study some
asymptotic properties of the unitary groups $(\e^{-\im t H} ; t \in
\R)$ and $(\e^{-\im t H_g} ; t \in \R)$.  In particular, we are
interested in the question whether there is transmission from the
lower to the upper sheet and vice versa. As noted by Percy Deift 
(private communication), this amounts to the question
\begin{center}
  \emph{``When I shout on the lower plane, will I be heard on the
    upper plane?''}
\end{center}
\begin{figure}[h]
  \centering
    \begin{picture}(80,30)
      \includegraphics{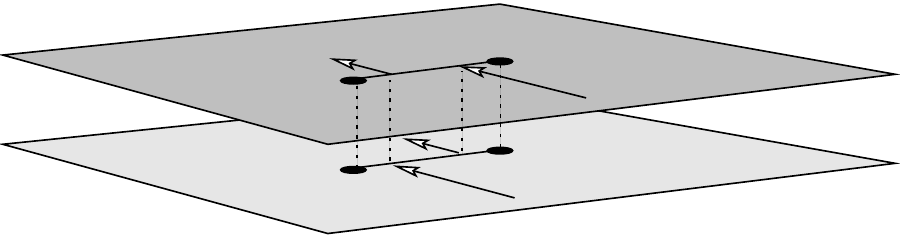}
      \put(-8,28){\vector(-3,-1){30}}
      \put(-8,28){\vector(-3,-2){30}}
      \put(-6,28){$q_+$}
      \put(-87,26){\vector(3,-1){30}}
      \put(-87,26){\vector(3,-2){30}}
      \put(-87,28){$q_-$}
      \put(-10,0){$M$}
      \put(0,5){$M_\low$}
      \put(0,15){$M_\upp$}
    \end{picture}
    \caption{The double covering $M$ with two branch points $q_-$ and
      $q_+$, and the straight line cut $\Gamma$ between $q_-$ and
      $q_+$.  If one arrives from the lower sheet $M_\low$ from below
      (in the picture from the right), then one continues on the upper
      sheet $M_\upp$ and vice versa. Points along the dashed lines are
      identified as explained above.}
\label{fig:1}
\end{figure}
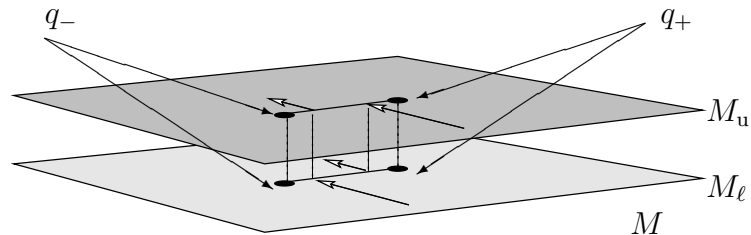
For the comparison dynamics (with two scattering channels) we take the
free Laplacian on two copies of $\R^2$ which we may imagine to lie one
atop of the other. In other words, we consider the Hilbert space
$\HH_0 = \Lsqr {\R^2} \oplus \Lsqr{\R^2}$ and we let $H_0$ denote the
direct sum of two copies of the self-adjoint Laplacian in
$\Lsqr{\R^2}$,
\begin{equation*}
    H_0 = H_{0,\low} \oplus H_{0,\upp}, 
\end{equation*}
where the indices $\low$ and $\upp$ mean ``lower'' and ``upper,''
respectively. $H_0$ is (purely) absolutely continuous. With a natural (unitary) identification $J \colon \HH_0
\to \HH$ the wave operators
\begin{equation*}
   W_\pm(H,H_0,J) =\slimpm e^{\im tH} J e^{-\im tH_0},  
\end{equation*}
exist, are complete, and isometric, as will be seen in
\Sec{wave-ops-euclid}. Since also $H$ is absolutely continuous 
the wave operators $W_\pm(H,H_0,J)$ are in fact unitary. 
Writing $J = J_\low \oplus J_\upp$, the
\emph{channel wave operators} $W_\pm(H, H_{0,\low}, J_\low)$ and
$W_\pm(H, H_{0,\upp}, J_\upp)$ are given by
\begin{equation*}
     W_\pm(H, H_{0,k}, J_k)
     =\slimpm e^{\im tH} J_k e^{-\im tH_{0,k}}, 
     \qquad k \in \{\low, \upp\}. 
\end{equation*}
Note that $f \in \ran W_+(H, H_{0,\upp}, J_\upp)$ means that there
exists $h \in \HH_\upp$ such that
\begin{equation*}
  \norm{\e^{-\im t H} f - J_\upp \e^{-\im tH_{0,\upp}} h } \to 0, 
  \qquad t \to +\infty; 
\end{equation*}
in particular, $\e^{-\im tH}f$ is asymptotically in the \emph{upper}
sheet, as $t \to +\infty$.  This leads to the question whether states
which come in on the lower sheet will also go out on the lower sheet,
or whether there are states which change sheets as $t$ goes from
$-\infty$ to $+\infty$. We construct, indeed, states that move from
the lower to the upper sheet, up to a small error.  It follows that
there is non-zero transmission between the upper and the lower sheets
of $\MM$, or, in the terminology of~\cite{hpw:14}, that the upper and
the lower channels are open. By symmetry there is also transmission
from the upper to the lower sheet; since it is more or less trivial
that there is transmission within the two sheets we find that all
scattering channels are open one to another. This is stated as
\Thm{1.9}.

We next ask whether the scattering channels remain open when the
Euclidean metric $\gEucl$ on $M$ is replaced with a more general
metric $g$ on $M$ which is close to $\gEucl$ at infinity in the sense
of~\cite{hpw:14}.  The corresponding assumptions concern, in
particular, the harmonic radius~\cite{anderson-cheeger:92, hpw:14})
and the injectivity radius of $(M,g)$, and the difference of the
Riemannian metrics $\gEucl$ and $g$ in a suitable distance function.
Here we profit in several ways from the fact that the geometry of
$\MM$ is so simple.  We require that the metrics $g$ and $\gEucl$ be
\emph{quasi-isometric} in the usual sense (cf.~\Def{2.2}), and we
assume a global bound on the curvature of $(M,g)$. Under additional
assumptions on $g$, expressed in terms of the distance
$\tilde{d}_1(\gEucl,g)$ in eqn.~\eqref{eq:2.6}, \Thm{2.3} states that
the wave operators
\begin{equation}
   \label{eq:wave-operators.E-g}
   W_\pm(H_g, H_\gEucl, I_g) :=\slimpm \e^{\im tH_g} I_g \e^{-\im tH_\gEucl} 
\end{equation} 
exist and are complete, where $H_\gEucl = H$ is the Laplacian of
$(M,\gEucl)$, $H_g$ is the Laplacian of $(M,g)$, and $I_g$ is the
natural identification between $\Lsqr{M,\gEucl}$ and $\Lsqr{M,g}$; as 
was mentioned earlier, $H_\gEucl$ is purely absolutely continuous. 

In \Thm{2.3}, smallness of the perturbation is only required at
infinity. In contrast, for the question of openness of the scattering
channels the deviation of $g$ from $\gEucl$ has to satisfy a global,
quantitative smallness condition. Then \Thm{2.4} establishes the
strong convergence of the scattering operators
\begin{equation*}
    S(H_{g_\eps}, H_0, I_{g_\eps} J) 
    :=  \bigl(W_+(H_{g_\eps}, H_0, I_{g_\eps}J)\bigr)^* 
                \circ W_-(H_{g_\eps}, H_0, I_{g_\eps}J) 
\end{equation*} 
to $S(H, H_0, J)$ for a sequence of metrics $g_\eps$ on $M$ tending to
$\gEucl$ as $\eps \downarrow 0$.  In \Cor{2.5} we then obtain the
openness of all scattering channels for small $\eps$.

The paper is organized as follows. In \Sec{wave-ops-euclid}  we
introduce most of our notation and we discuss some basic spectral
properties of the manifold $\MM = (M,\gEucl)$, deferring the details 
and proofs to Appendix A. We then turn to
scattering for the pair $(H, H_0)$ where we establish existence and
completeness of the wave operators. The technically difficult part of 
\Sec{wave-ops-euclid} concerns the construction of a wave packet that
comes in from infinity on the lower sheet and moves out to infinity on
the upper sheet. Here we use ideas from En{\ss}' theory of scattering
and stationary phase estimates to construct states that pass between
the branch points $\Qpm$ at time $t = 0$ at high speed, and which are
essentially localized to a double cone.

In \Sec{pert-metric} we consider metrics $g$ on $M$ that are close
(or, at least, close at infinity) to the Euclidean metric $\gEucl$. In
essence, we only have to write down what the basic definitions and
results of~\cite{hpw:14} mean in the present context. We then find
simple conditions for the existence and completeness of the wave
operators~\eqref{eq:wave-operators.E-g} as well as for a non-trivial
interaction between the scattering channels for $(M,g)$.

The main results of \Sec{pert-metric} are illustrated in
\Sec{examples} by a simple class of metrics on $M$, namely metrics $g
= g_f$ that come from the graph of smooth functions $f$ on $M$.  It
turns out that it is fairly easy to indicate conditions on $f$ so that
the metric $g_f$ satisfies the requirements of \Thm{2.3}.  We finally
discuss branched coverings with more than two sheets and corresponding
generalizations of the present results.

 The paper comes with three appendices; the first two of them
 are mainly included for the convenience of the reader.
 \App{spectral-properties} is
 devoted to self-adjoint extensions, compactness and spectral
 properties of the Laplacian with metrics $\gEucl$ and $g$.
 As for the absolute continuity of $H_\gEucl$ and $H_g$,
 we mainly refer to some work of Donnelly~\cite{donnelly:99}
 and Kumura~\cite{kumura:10, kumura:13}.
 
In \App{stationary-phase} we recall a basic estimate from stationary
phase theory to establish an estimate on the localization error for
the Schr\"odinger evolution. More precisely, for suitably chosen
initial data $u_0$ in the Schwartz space $\Schwartz {\R^2}$ we
multiply $u(t) := \e^{\im t \Delta} u_0$ by a cut-off function $\chi$
and obtain estimates for $\nabla \chi \cdot \nabla u(t)$ and $(\Delta
\chi) u(t)$ in the $\Lsqrspace$-norm.

\App{lower-bound-inj-rad} is devoted to lower bounds for the
injectivity radius of $(M,g)$ where the metric $g$ on $\R^2$ or on $M$
is close to the Euclidean metric. Starting from a comparison result of
M\"uller and Salomonsen~\cite{mueller-salomonsen:07} we obtain
``local'' versions by means of cut-offs and extension theorems,
proceeding from $\R^2$ via $\R^2 \setminus \{(0,0)\}$ to $M$.

We conclude the introduction with a few remarks concerning the literature.
The paper \cite{hpw:14} and the literature quoted there give
a partial overview of Riemannian scattering on manifolds with ends. 
Recent progress in this direction can be found in G\"uneysu and Thalmaier
~\cite{gueneysu-thalmaier:17}. The specific case of manifolds with
branch points has been studied in recent years under various aspects and
our results have some overlap with the work of Hillairet and others;
cf.~\cite{hillairet:10}  and \cite{ford-hassell-hillairet:15}.
There is a connection between the analysis of the Aharonov-Bohm effect
in Quantum Mechanics and branched coverings of Euclidean space;
cf.~\cite{bnhho:09}. Scattering for magnetic Schr\"odinger operators
with two magnetic point charges has been studied in a number of papers; as
an example, we mention Ito and Tamura \cite{ito-tamura:01} which has
some connection with our investigations. 

\subsection*{Acknowledgements}

The authors thank Luc Hillairet (Univ.\ d'Orl\'eans) for an 
interesting discussion and comments. Rainer Hempel would like 
to express his gratitude to Brian Davies
(King's College, London), Percy Deift (Courant Institute, New York),
Ira Herbst (Univ.\ of Virginia, Charlottesville), Barry Simon
(Caltech, Pasadena), and Larry Thomas (Univ.\ of Virginia) for
valuable discussions and suggestions concerning the matter of the
present paper.

%
\section{Wave operators for the Euclidean metric}
\label{sec:wave-ops-euclid}  
%

Let us begin with some notation. As far as general notation
for self-adjoint operators $T$ in a Hilbert space $\HH$ is concerned we
mostly follow~\cite{kato:66} and~\cite{reed-simon-1}.  In particular,
we let $\HH_\ac(T)$ denote the absolutely continuous subspace of $\HH$
associated with $T$, and $P_\ac(T)$ the orthogonal projection onto
$\HH_\ac(T)$.  For the general formal setup of multi-channel
scattering we refer to Section~4 of~\cite{hpw:14} and the literature
quoted there.  Since the model studied in the present paper is so
simple, we develop most notions in multi-channel scattering directly
as we go along.

Let $M$ be defined as in the Introduction. We then denote the points
$\Pt$ of $M$ by $((x,y), \low)$ or $((x,y),\upp)$ where $(x,y) \in
\R^2$ and ``$\low$'' means ``lower,'' ``$\upp$'' means ``upper''.
This works for all points of $M$ with the exception of the points with
$-1 < x < 1$ and $y = 0$; note that these exceptional points form a
set of measure zero. With $\gEucl$ denoting the metric tensor $\gEucl
= (\delta_{ij})$ we obtain the Riemannian manifold $\MM :=
(M,\gEucl)$.  For the remainder of this section we will be cavalier
about the distinction between $M$ and $\MM=(M,\gEucl)$ and we will
mostly write $M$.  For two points $\Pt_1$, $\Pt_2 \in M$ the
(geodesic) distance is then given by
\begin{equation}
  \label{eq:1.1}
  \dist(\Pt_1,\Pt_2) 
  := \inf \set{\abs \gamma}{\gamma(0) = \Pt_1, \gamma(1) = \Pt_2} 
\end{equation}
where $\gamma \colon [0,1] \to M$ is a rectifiable curve and $\abs
\gamma$ denotes the length of $\gamma$.  It will be useful to extend
the definition of distance to the branch points $\Qm$ and $\Qp$.  The
infimum in~\eqref{eq:1.1} is attained either for a straight line
segment connecting $\Pt_1$ and $\Pt_2$ or for (the union of) two straight line
segments that meet at one of the branch points.  E.g., if $\Pt_1 =
((0,y), \low)$, $\Pt_2 = ((0,-y),\low)$ with $y > 0$, then
$\dist(\Pt_1,\Pt_2) = \dist(\Pt_1, \Qm) + \dist(\Qm,\Pt_2) = 2 \sqrt{1
  + y^2}$ (see \Fig{2} left).
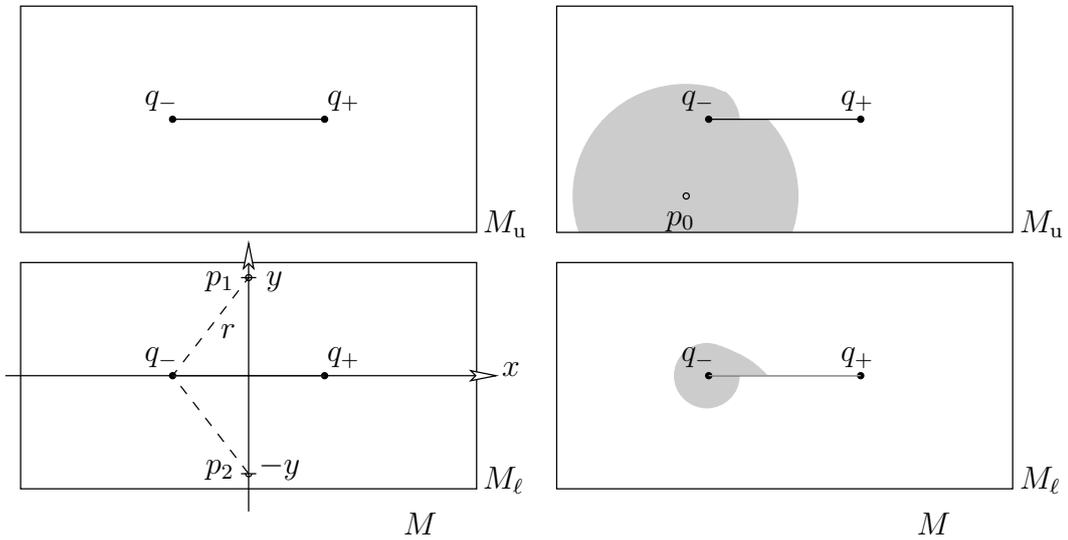
\begin{figure}[h]
  \centering
  \begin{minipage}{0.45\linewidth}
    \begin{picture}(80,75)
      \includegraphics{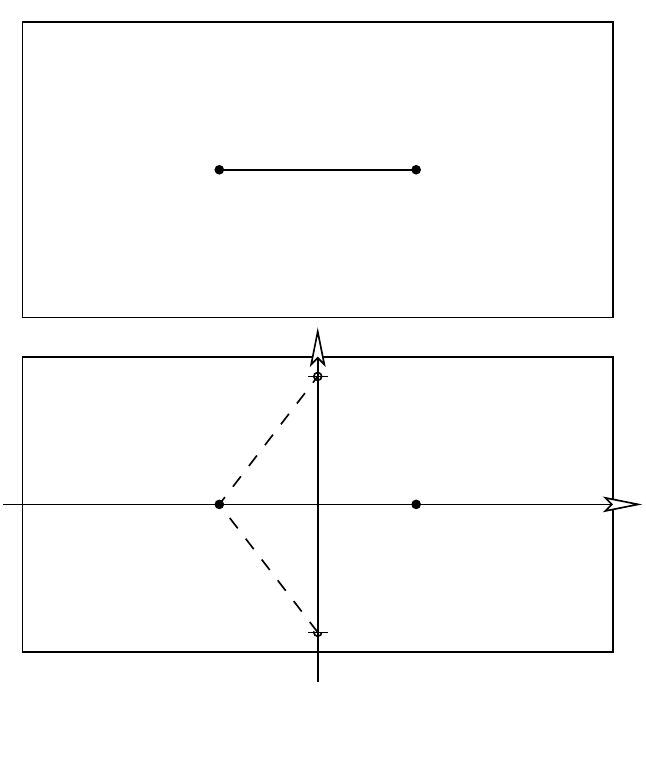}
      \put(-47,62){$q_-$}
      \put(-47,28){$q_-$}
      \put(-37,31){$r$}
      \put(-23,62){$q_+$}
      \put(-23,28){$q_+$}
      \put(-31,38){$y$}
      \put(-32,13){$-y$}
      \put(0,26){$x$}
      \put(-39,38){$p_1$}
      \put(-39,13){$p_2$}
      \put(-2.5,45){$M_\upp$}
      \put(-2.5,11){$M_\low$}
      \put(-13,5){$M$}
    \end{picture}
  \end{minipage}
  \begin{minipage}{0.45\linewidth}
    \begin{picture}(80,75)
      \includegraphics{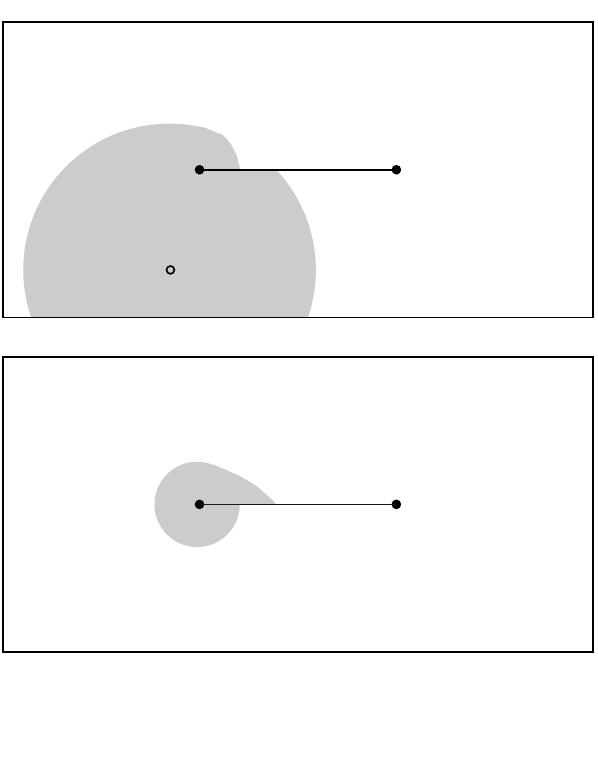}
      \put(-44,62){$q_-$}
      \put(-44,28){$q_-$}
      \put(-46,46){$p_0$}
      \put(-23,62){$q_+$}
      \put(-23,28){$q_+$}
      \put(0.5,45){$M_\upp$}
      \put(0.5,11){$M_\low$}
      \put(-13,5){$M$}
    \end{picture}
  \end{minipage}
    \caption{Left: The distance between $p_1$ and $p_2$ is $2r=2\sqrt{1+y^2}$.
      Right: The shaded area is a disc $B_r(p_0)$ with points in
      both sheets.}
\label{fig:2}
\end{figure}

For a point $\Pt_0 \in M$, we denote the (geodesic) disc of radius 
$r > 0$ and center $(x_0, y_0)$ by $B_r(\Pt_0)$, i.e., 
\begin{equation}
  \label{eq:1.2}
   B_r(\Pt_0) = \set{\Pt \in M}  {\dist(\Pt, \Pt_0) < r};  
\end{equation}
such discs may or may not contain points in both sheets (see \Fig{2}
right), and they may even contain pairs of points $(p, p')$ with the
same $(x,y)$-coordinates and $p$ in the lower, $p'$ in the upper
sheet.  A disk $B_r(p_0)$ will be ``single-valued'' if and only if $r
\le \min\{\dist(p_0, q_+), \dist(p_0,q_- \}$. In the extreme case of
$p_0 \in \{q_+, q_-\}$ and $0< r \le 2$ the disk $B_r(p_0)$ will just
be a double covering of the punctured disk $\set{(x,y) \in \R^2}{0 <
  x^2 + y^2 < r^2 }$.  
The Riemannian manifold $\MM$ is not (geodesically) complete.

In order to define the Laplacian $H$ of $\MM$, we consider the Hilbert
space $\HH := \Lsqr \MM$ with scalar product denoted by $\scapro \cdot
\cdot$, and the Sobolev space $\Sobn \MM$, given as the completion of
$\Cci M$ with respect to the norm $\norm[1] \cdot$ defined by
\begin{equation}
  \label{eq:1.3}
    \normsqr[1] \psi
    := \int_M \abssqr{\psi(x)} + \abssqr{\nabla \psi(x)} \dd x, 
    \qquad \psi \in \Cci M. 
\end{equation}  
Then $H$ is defined as the unique self-adjoint operator satisfying
$\dom(H) \subset \Sobn \MM$ and
\begin{equation}
  \label{eq:1.4}
  \scapro{Hu} v 
  = \int_M \nabla u \cdot \nabla {\conj v} \dd x, 
  \qquad u \in \dom(H), \,\,\, v \in \Sobn \MM.
\end{equation}
It is easy to see (cf.\ \App{spectral-properties}) that $\Sobn \MM$
coincides with the Sobolev space $\Sob \MM = \SobW[2] 1 \MM$,
consisting of all functions in $\Lsqr \MM$ that have first order
distributional derivatives in $\Lsqr \MM$. Hence the Laplacian on
$\Cci M$ has only one self-adjoint extension with form domain
contained in $\Sob \MM$.  However, the Laplacian is \emph{not}
essentially self-adjoint on $\Cci M$.  Basic spectral properties of
$H$ are also discussed in \App{spectral-properties}; in particular,
$H$ is purely absolutely continuous with $\spec H = \spec[ac] H =
[0,\infty)$.

We next consider the Rellich compactness property.  For the proof we
refer to \Prp{A.2} in \App{spectral-properties}.
\begin{lemma}
  \label{lem:1.2}
  For $R > 0$, let $\chi_R$ denote the characteristic function of $M_R
  = B_R(\Qm) \cup B_R(\Qp) \subset M$ (see \Fig{2}).  Then the mapping
  $\Sob \MM \ni u \mapsto \chi_R u \in \Lsqr \MM$ is compact.
\end{lemma}

We now turn to scattering theory and introduce the comparison dynamics
for the scattering channels associated with the two sheets (and two
infinities) of $M$.

Let $M_0 := \R^2 \uplus \R^2 = \R^2 \times \{\low, \upp\}$ denote the
disjoint union of two copies of the Euclidean plane $\R^2$, and write
$M_{0,\low} = \R^2 \times \{\low\}$, $M_{0,\upp} = \R^2 \times
\{\upp\}$.  We then let $\HH_0 = \Lsqr {M_0,\gEucl} = \Lsqr{\R^2}
\oplus \Lsqr{\R^2}$ Moreover, we let $H_0$ denote the Laplacian on
$M_0$. To fix the notation, let $A_0$ denote the (unique) self-adjoint
extension of $-\Delta$ on $\Cci {\R^2}$. We may then write $H_0 =
H_{0,\low} \oplus H_{0,\upp}$ where $H_{0,\low}$ and $H_{0,\upp}$ act
as $A_0$ in $\Lsqr{M_{0,\low},\gEucl}$ and in $\Lsqr{M_{0,\upp}}$,
respectively.

We denote the straight line segment in $\R^2$ connecting the points
$\Qpm$ as $\Gamma$,
\begin{equation}
  \label{eq:1.5}
    \Gamma := [-1,1] \times \{0\} \subset \R^2, 
\end{equation} 
a set of measure zero.  There is a natural embedding $\iota \colon
(\R^2 \setminus \Gamma) \times \{\low, \upp\} \to M$, $\iota =
(\iota_\low, \iota_\upp)$, where $\iota_\low$ maps the point 
$((x,y), \low) \in M_{0,\low} \setminus \Gamma$ to $((x,y),\low) \in M$, 
 and similarly for $\iota_\upp$. 
 The embedding $\iota$ induces a unitary mapping $J
\colon \HH_0 \to \HH$ where $J = J_\low \oplus J_\upp$ in an obvious
manner (and with a slight abuse of notation).  $J_\low$ maps functions
$f \in \Lsqr{M_{0,\low},\gEucl}$ to the same function on the lower sheet of $M$
and extends them by zero to all of $M$, and similarly for $J_\upp$.
We then have:

\begin{proposition}
  \label{prp:1.4}
  The wave operators
  \begin{equation}
    \label{eq:1.6}
    W_\pm(H,H_0,J) =\slimpm \e^{\im tH} J \e^{-\im tH_0} 
  \end{equation} 
  exist and are unitary. 
\end{proposition}

\begin{remark} 
  As is often the case in two Hilbert space
    scattering~\cite{kato:67, reed-simon-3}, there is a certain
    arbitrariness in the choice of the mapping $J$. 
    By local compactness, the same wave
    operators and the same results would be obtained with $J$ replaced
    by $(1 - \chi_R)J$, for some $R > 0$, or by $(1 - \phi)J$
    with an  arbitrary $\phi \in \Cci{\R^2}$. 
 \end{remark}

\begin{proof}[Proof of \Prp{1.4}]
  We decouple both $H$ and $H_0$ by Dirichlet boundary conditions
  along two circles defined as follows. Let $\calC_2 := \set{(x,y) \in
    \R^2}{x^2 + y^2 = 4}$, $\calC_2' := \calC_2 \times \{\low,\upp\}
  \subset M_0$, and $\calC_2'' := \iota( \calC_2') \subset
  M$. Introducing Dirichlet boundary conditions on $\calC_2'$ and on
  $\calC_2''$ decomposes $H_0$ into a direct sum of four operators
  while $H$ is decomposed into a direct sum of three operators. More
  precisely, we introduce the following three ``building blocks:'' in
  the plane $\R^2$, we have the Dirichlet Laplacian $h_\interior$ on
  the disc of radius $2$ and the Dirichlet Laplacian $h_\ext$ on the
  exterior of this disc.  Furthermore, defining
  \begin{subequations}
    \label{eq:1.7}
    \begin{align}
      \label{eq:1.7a}
      M_{0,\ext} 
      &:= \set{ (x,y) \in \R^2} {x^2 + y^2 > 4} \times \{\low, \upp\},\\
      \label{eq:1.7b}
      M_\ext 
      & := \iota(M_{0,\ext)}), 
      \quad M_\interior := M \setminus \clo M _\ext
    \end{align}
  \end{subequations}
  we denote by $H_\interior$ the Dirichlet Laplacian of
  $M_\interior$. Note that $M_\interior$ is a branched covering with
  two sheets of the punctured disc $\set{(x,y) \in\R^2}{x^2 + y^2
    < 4} \setminus\{q_+, q_-\}$.  We then write
  \begin{subequations}
    \label{eq:1.8}
    \begin{align}
      \label{eq:1.8a}
      H_{0,\dec}
      & := (h_\interior , \low) \oplus (h_\ext, \low) 
        \oplus  (h_\interior , \upp) \oplus (h_\ext, \upp),\\
      \label{eq:1.8b}
      H_{\dec} 
      & := H_\interior \oplus (h_\ext , \low) \oplus (h_\ext , \upp);
    \end{align}
  \end{subequations}
  note that $h_\ext$ is purely absolutely continuous while
  $h_\interior$ and $H_\interior$ (by \Lem{1.2}) have compact
  resolvent.

  It is well-known (\cite{birman:63, deift-simon:76,
    hpw:14}) that the wave operators
  \begin{equation}
    \label{eq:1.9}
    W_\pm(H_{0,\dec},H_0) =\slimpm  \e^{\im tH_{0,\dec}}  \e^{-\im tH_0} 
  \end{equation}
  exist, are complete, and isometric with initial subspace
  $\HH_\ac(H_0) = \HH_0$ and final subspace $\HH_\ac(H_{0,\dec}) =
  \Lsqr{M_{0,\ext},\gEucl}$.
 Similarly, it can be shown by standard methods (cf.~\cite{deift-simon:76, 
  hempel-weder:93, hpw:14}), that the wave operators
  \begin{equation}
    \label{eq:1.10}
    W_\pm(H, H_\dec) =\slimpm \e^{\im tH} \e^{-\im tH_\dec} P_\ac(H_\dec)  
  \end{equation}
  exist, are complete, and partially isometric with initial subspace
  $\HH_\ac(H_\dec) = \Lsqr{M_\ext,\gEucl} = \HH_\ac(H_{0,\dec}) =
  \Lsqr{M_{0,\ext},\gEucl}$ and final subspace $\HH_\ac(H) = \Lsqr {M,\gEucl}$.
  Finally, the wave operators
  \begin{equation}
    \label{eq:1.11}
    W_\pm(H_\dec, H_{0,\dec},J)
    =\slimpm \e^{\im t H_\dec} J \e^{-\im t H_{0,\dec} } P_\ac(H_{0,\dec})
  \end{equation}
  simply act as the identity on $\Lsqr{M_\ext,\gEucl}$, and as the zero
  operator on $\Lsqr{M_\interior,\gEucl}$.  Therefore, they exist and are
  complete.
 It is now clear that the wave operators $W_\pm(H,H_0,J)$ exist and are unitary. 
\end{proof}

With $J = J_\low \oplus J_\upp$ and $H_{0,\low}$ as defined above, we
furthermore see that the \emph{channel wave operators}
\begin{equation}
  \label{eq:1.12}
  W_\pm(H, H_{0,\low}, J_\low)
  = \slimpm \e^{\im tH} J_\low \e^{-\im tH_{0,\low}}
\end{equation}
(and, analogously, $W_\pm(H, H_{0,\upp}, J_\upp$) exist and are isometric  with 
\begin{equation}
  \label{eq:1.13}
  \ran W_\pm(H, H_{0,\low}, J_\low) \oplus \ran W_\pm(H, H_{0,\upp}, J_\upp)
  =  \ran W_\pm(H, H_0, J)
  = \HH_\ac(H); 
\end{equation} 
recall that $f \in \ran W_+(H, H_{0,\low}, J_\low)$ means that there
exists $g \in \Lsqr{M_{0,\low},\gEucl}$ such that
\begin{equation}
  \label{eq:1.14}
  \norm{\e^{-\im t H} f - J_\low \e^{-\im t H_{0,\low}} g} \to 0, 
  \qquad t \to \infty; 
\end{equation} 
in particular, $\e^{-\im t H} f$ is asymptotically on the lower sheet
for $t \to \infty$.  Eqn.~\eqref{eq:1.13} establishes two orthogonal
decompositions of $\HH_\ac(H) = \Lsqr {M,\gEucl}$, one for the plus-sign
and another one for the minus-sign. We will see later on
(cf.~\Lem{1.7.2nd.part}) that these two decompositions are in fact
different.

\begin{remark}
  \label{rem:1.5}
  Let us note that $H_0 = H_{0,\low} \oplus H_{0,\upp}$ provides a
  reference operator for $H$ in the sense of~\cite[Def.~4.7]{hpw:14}
  with two channels.  Strictly speaking, branch points like $\Qpm$ are
  not directly included in the framework used
  in~\cite{hpw:14}. However, this technical difficulty is easy to
  resolve: we might just take each of the sets $B_{1/2}(\Qpm)$ as an
  end, albeit an end which does not participate in the scattering
  process since the Dirichlet Laplacian of $B_{1/2}(\Qpm)$ has compact
  resolvent by \Lem{1.2}. The possibility of allowing such ``dead
  ends'' is described in Remark~4.4 of~\cite{hpw:14}. We thus have
  (formally) a manifold with 4 ends, with two ends given by a copy of
  $\R^2 \setminus B_2(0)$ and another two ends given by
  $B_{1/2}(\Qpm)$.
\end{remark}
\medskip

It is a major goal in scattering theory to obtain information on the
\emph{scattering operator}
\begin{equation}
   \label{eq:scattering-operator}
  S = S(H,H_0,J) 
  := \bigl(W_+(H,H_0,J)\bigr)^*\circ W_-(H,H_0,J) \colon \HH_0 \to \HH_0, 
\end{equation}
a unitary operator, and the closely related \emph{scattering matrix}
$(S_{ij})_{i,j \in \{\low, \upp\}}$, with
\begin{equation}
  \label{eq:scattering-matrix}
  S_{ij} 
  := \bigl(W_+(H, H_{0,i}, J_i)\bigr)^* \circ W_-(H, H_{0,j}, J_j)
                           \colon \Lsqr{M_{0,j},\gEucl} \to \Lsqr{M_{0,i},\gEucl},
\end{equation}
for $i, j \in \{\low, \upp\}$. We will show that the four components
of $(S_{ij})$ are non-zero which yields the openness of all scattering
channels. 

The following lemma establishes the existence of a state $w_0$ for
which $\e^{-\im tH}w_0$ is asymptotically in the lower sheet for $t
\to -\infty$ and in the upper sheet for $t \to +\infty$, up to small
errors. Recall that $A_0$ denotes the self-adjoint extension of the Laplacian
on $\R^2$. We then have:

\begin{lemma}
  \label{lem:1.7}
  For $\eps > 0$ given, there exist $w_0 \in \Lsqr {M,\gEucl} \cap
  \Contspace[\infty](M)$, $v_0 \in \Schwartz{\R^2}$, and $t_0 \ge 0$
  such that the following estimates hold:
  \begin{equation}
    \label{eq:lem.1.7.1}
    \norm{ \e^{-\im tH} w_0 - J_\upp \e^{-\im t A_0} v_0} < \eps \norm{w_0}, 
    \qquad t \ge t_0, 
  \end{equation}
  and 
  \begin{equation}
    \label{eq:lem.1.7.2}
    \norm{ \e^{-\im tH} w_0 - J_\low \e^{-\im t A_0} v_0} < \eps \norm{w_0}, 
    \qquad  t \le -t_0.
  \end{equation}
\end{lemma} 

In the proof of \Lem{1.7} we basically construct a state $v_0 \in
\Lsqr{\R^2}$ which passes at high speed between the points $\Qpm$
under the evolution determined by $\e^{-\im tA_0}$ (up to small
errors) and whose spreading can be controlled by stationary phase
estimates, for $|t|$ large.  Note that we have complete control of the
unitary group $(\e^{-\im t A_0} \> ; \> t \in \R)$, acting in
$\Lsqr{\R^2}$, while we know much less about $(\e^{-\im tH} \> ; \> t
\in \R)$, acting in $\Lsqr {M,\gEucl}$.  By a simple lifting, $v_0$ is
transformed into a function $w_0$ on $M$.  Here we wish to gain
information on the evolution of $\e^{-\im t H} w_0$ from the
properties of $\e^{-\im t A_0} v_0$ using the fact that both operators
act locally as the Laplacian.

Recall that $H_{0,\upp}$ and $H_{0,\low}$ denote the self-adjoint
Laplacian in $\Lsqr{M_{0,\upp},\gEucl}$ and in $\Lsqr{M_{0,\low},\gEucl}$,
respectively.  We let $\FF$ denote the Fourier transform on the
Schwartz spaces $\Schwartz{\R^d}$ for $d \in \N$.  It is well known
that $\FF$ acts bijectively on $\Schwartz{\R^d}$ and extends to a
unitary map $\FF \colon \Lsqr{\R^d} \to \Lsqr{\R^d}$.


Our construction starts with a function $u_0 \in \Schwartz{\R^2}$ of
the form $u_0 = u_0(x,y)$, given as the product of two functions
$\psi_1= \psi_1(x)$ and $\psi_2 = \psi_2(y)$ enjoying certain
properties, which we describe now.

Let $\eps \in (0,1)$ be given and let $\eps' := \eps/5$.  We first
pick a function $\phi_1 \in \Cci \R$ of norm 1 and we let $\psi_1 :=
\FF^{-1} \phi_1 \in \Schwartz \R$ where we assume that
\begin{equation}
  \label{eq:1.18}
  \norm{\chi_{(-\frac14,\frac14)}\psi_1} > 1 -\eps'.
\end{equation}   
We let $a = a_\eps > 0$ be such that $\supp \phi_1 \subset (-a,a)$.
Next, let $\phi_2 \in C_c^\infty (0,1)$, of norm 1 again, and let
$\psi_2 := \FF^{-1} [ \phi_2(.-s)] \in \Schwartz \R$, where $s > 0$
will be chosen later.  Let
\begin{equation}
  \label{eq:1.19}
  u_0 = u_0(x,y) = \psi_1(x) \psi_2(y), 
  \qquad (x,y) \in \R^2.  
\end{equation}
Then $u_0 \in \Schwartz{\R^2} \subset \dom(A_0)$ and $u(t) := \e^{-\im
  tA_0} u_0$ is a classical solution of the initial value problem for
the Schr\"odinger equation in $\Lsqr{\R^2}$, i.e.,
\begin{equation*}
  \dot u(t)
  = -\im A_0 u(t) \quadtext{for} t \in (0,\infty),
  \qquad u(0) = u_0.        
\end{equation*}
We write
\begin{equation*}
  Q_{s,t} := 
  \begin{cases} 
    (-st, st) \times (st, \infty), &  t > 0, \\
    (st, -st) \times (-\infty, st), & t < 0,
  \end{cases}
\end{equation*} 
for $s > 0$, and we let $\chi_{s,t}$ denote the characteristic
function of $Q_{s,t}$.  \Lem{a.2} implies that for any $m \in \N$
there exists a constant $\wt{c}_m \ge 0$ such that
\begin{equation*}
  \norm{(1 - \chi_{s,t}) \e^{-\im t A_0} u_0}  \le   \wt{c}_m (1 + st)^{1 - 2m}, 
  \qquad s \ge 2a, \quad t > 0, 
\end{equation*}
so that for $s \ge 2a$ and $t$ large, $t \ge t_0$ say,  
\begin{equation} 
  \label{eq:1.19a}
  \norm{( 1 - \chi_{s,t}) \e^{-\im t A_0} u_0} \le \eps'. 
\end{equation}
\medskip

Now let
\begin{equation}
  \label{eq:1.20}
  \Omega := \bigset{(x,y) \in \R^2}{ \abs x < 1/2 + \abs y}, 
\end{equation}
let $\chi_\Omega$ denote the characteristic function of $\Omega$,
and, finally,
\begin{equation}
  \label{eq:1.21}
  \chi := j_{\frac14} * \chi_\Omega, 
\end{equation}
where $(j_\delta)_{\delta > 0}$ is the kernel of the usual Friedrichs
mollifier on $\R^2$; in particular, $0 \le j_\delta \in \Cci {\R^2}$
with support in the closed disc of radius $\delta$, and $\int j_\delta
= 1$.  Also let $\XX$ denote the support of $\chi$ and $X$ the
characteristic function of $\XX$, i.e, $X = \chi_\XX$. Note that
$\chi$ is independent of $t$.

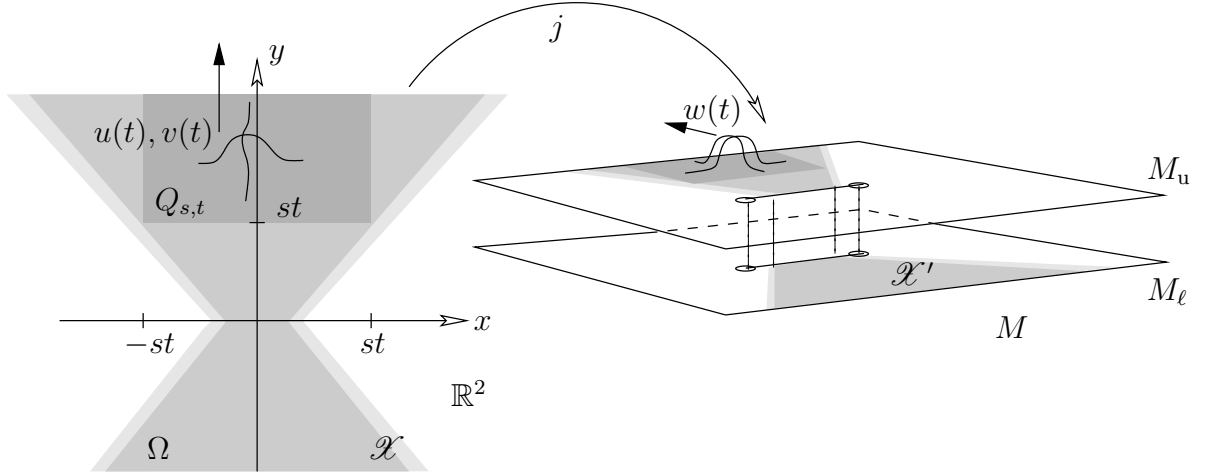
\begin{figure}[h]
  \centering
    \begin{picture}(60,60)
      \includegraphics{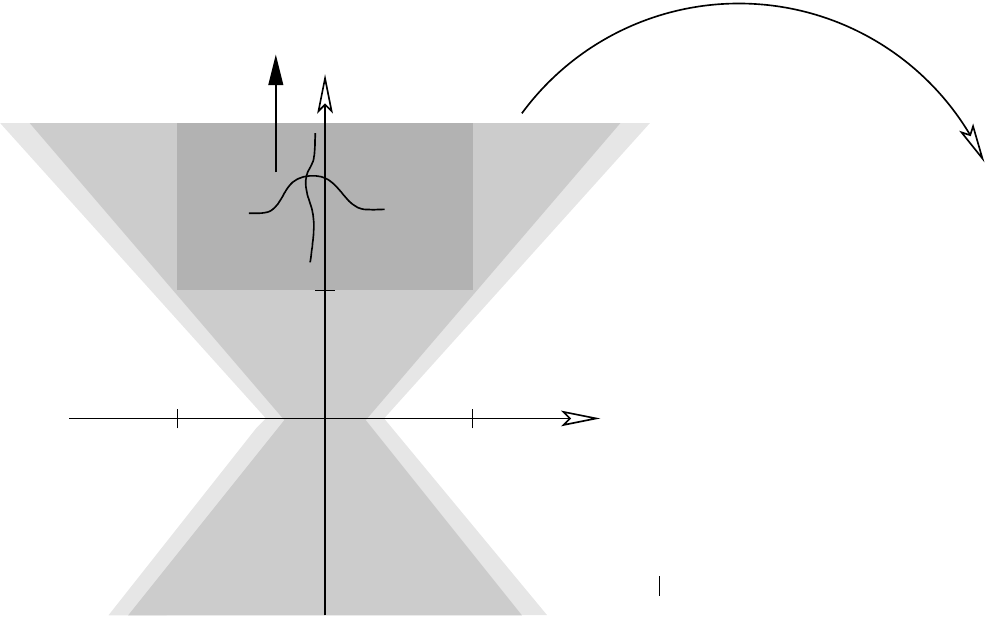}
      \put(-54,16){$st$}
      \put(-85,16){$-st$}
      \put(-65,34){$st$}
      \put(-89,44){$u(t),v(t)$}
      \put(-81,35){$Q_{s,t}$}
      \put(-82,2){$\Omega$}
      \put(-53,2){$\XX$}
      \put(-42,9){$\R^2$}
      \put(-39,19){$x$}
      \put(-66,55){$y$}
      \put(-29,58){$j$}
    \end{picture}
    \begin{picture}(90,60)
      \includegraphics{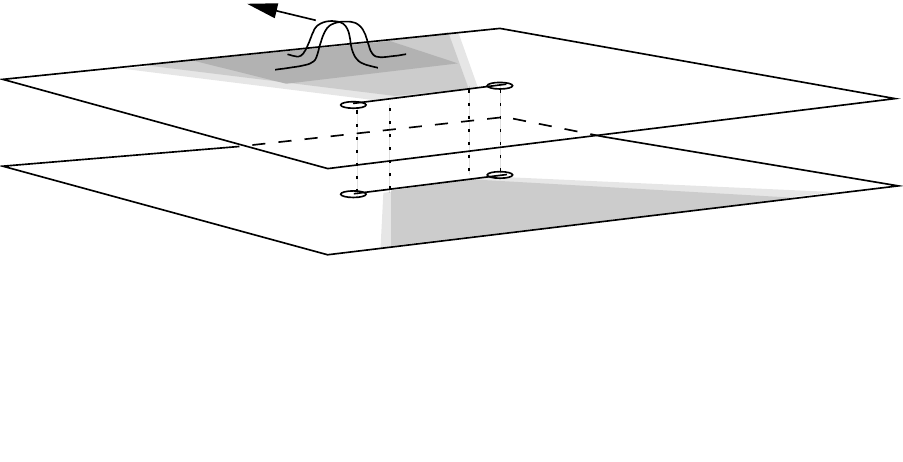}
      \put(-3,39){$M_\upp$}
      \put(-3,23){$M_\low$}
      \put(-23,18){$M$}
      \put(-37,25){$\XX'$}
      \put(-64,47){$w(t)$}
    \end{picture}
    \caption{\emph{Left:} the wave packet $u_0$ at time $0$ has speed
      $s$ in $y$-direction and is concentrated in $x$-direction near
      $x=0$; the support of the wave packet $u(t)=\e^{\im A_0 u_0}$ at
      time $t>0$ is essentially contained in the dark grey area
      $Q_{s,t}$.  Moreover, when considering the time evolution $v(t)$
      of the initial state $v_0= \chi u_0$ (with a cut-off function
      $\chi$ defined as a smooth version of the indicator function
      $\chi_\Omega$) with support in $\XX$, the deviation from $u(t)$
      is small.  \emph{Right:} the corresponding sets and the wave
      packet $w(t)$ corresponding to $v(t)$ on $M$.  The initial state
      here is $w_0=w(0)$.}
\label{fig:3}
\end{figure}
We next consider $v_0 := \chi u_0$ and observe that the (smooth)
function $v := \chi u$ is a solution of the
inhomogeneous initial value problem
\begin{equation}
  \label{eq:1.22}
  \dot v(t) = - \im A_0 v(t) +  f(t), 
  \qquad v(0) = v_0, 
\end{equation}
with $f = f(t) = f(x,y;t)$ given by
\begin{equation}
  \label{eq:1.23}
  f = - 2\im \nabla\chi \cdot \nabla u - \im u \Delta \chi.  
\end{equation}
We also have $\norm{v_0 - u_0} < \eps'$ and $\norm{v_0} > 1 - \eps'$.
Stationary phase estimates (cf.\ \Lem{a.3} in the appendix) imply that
there exists $s_0 \ge 0$ such that
\begin{equation}
  \label{eq:1.24}
  \int_{-\infty}^\infty \norm{f(\tau)} \dd \tau < \eps', 
  \qquad s \ge s_0. 
\end{equation} 
The solution $v = v(t)$ of eqn.~\eqref{eq:1.22} can be written as
\begin{equation}
  \label{eq:1.25}
  v(t) = \e^{-\im t A_0} v_0 + \int_0^t \e^{\im (t-\tau) A_0} f(\tau) \dd \tau.  
\end{equation} 
Notice that there is no reason to expect that for $t \ne 0$ the
individual terms $\e^{-\im t A_0} v_0$ and $\int_0^t \e^{\im (t-\tau)
  A_0} f(\tau) \dd \tau$ on the right-hand side of~\eqref{eq:1.25}
should vanish outside of $\XX$; it is only the sum of the two terms
which has support contained in $\XX$.  It is immediate from
eqn.~\eqref{eq:1.19a}, $\norm{u_0 - v_0} < \eps'$, and $\norm{u_0} =
1$ that
\begin{equation}
  \label{eq:1.26}
  \norm{(1 - \chi_{s,t}) \e^{-\im t A_0} v_0} \le 2 \eps',  
  \qquad \norm{\chi_{s,t} \e^{-\im t A_0} v_0} \ge 1 - 2\eps'. 
\end{equation}

We have now gathered all the information we need on $\e^{-itA_0} v_0$
and are ready for the proof of \Lem{1.7}.

\begin{proof}[Proof of \Lem{1.7}] 
  \begin{myenumerate}{\roman}
  \item
    In order to make the transition from $\R^2$ to $M$ we define a map
    $j \colon \XX \to M$ which assigns to $(x,y) \in \XX$ the point
    $((x,y),\low) \in M$ for $y < 0$, and the point $((x,y),\upp) \in
    M$ for $y > 0$.  The points in $\XX$ with $y = 0$ are mapped to
    the line segment where the lower and the upper sheets of $M$ are
    connected as we move in the direction of increasing values of
    $y$. Let $\XX' := j(\XX)$. For functions $\eta \colon \XX \to \C$,
    we obtain a lifting $\wt J \eta \colon \XX' \to \C$ defined by
    \begin{equation}
      \label{eq:1.27}
      (\wt J\eta)(j(x,y)) := \eta(x,y), 
      \qquad (x,y) \in \XX. 
    \end{equation} 
    We may extend $\wt J \eta$ by zero to all of $M$.  Obviously, we
    have $ w(t) := \wt J v(t) \in \dom(H)$ for $t > 0$ and $H (w(t)) =
    \wt J A_0(v(t))$.  Hence $w$ is a classical solution in $\Lsqr
    \MM$ of the initial value problem
    \begin{equation}
      \label{eq:1.28}
      \dot w(t) = -\im H w(t) + \wt Jf(t),
      \qquad w(0) = \wt Jv_0,  
    \end{equation} 
    so that
    \begin{equation}
      \label{eq:1.29}
      w(t)
      = \e^{-\im tH} \wt Jv_0 
           + \int_0^t \e^{\im (t-\tau)H} \wt J f(\tau) \dd \tau. 
    \end{equation}
    We conclude from eqns.~\eqref{eq:1.25} and~\eqref{eq:1.29} that
    \begin{align}
      \nonumber
      0 & = w(t) - \wt Jv(t)\\
      \label{eq:1.31}
        & = e^{-\im tH} \wt Jv_0 
           + \int_0^t \e^{\im (t-\tau)H} \wt Jf(\tau) \dd \tau 
           - \wt JX \e^{-\im t A_0} v_0 
           - \wt J X \int_0^t \e^{\im(t-\tau) A_0} f(\tau) \dd \tau,
    \end{align}
    whence
    \begin{equation}
      \label{eq:1.32}
      \bignorm[\Lsqr \MM] {\e^{-\im tH} \wt Jv_0 - \wt J X \e^{-it A_0} v_0}
      \le 2 \int_0^t \norm{f(s)} \dd s < 2\eps'.  
    \end{equation}
    We finally define $w_0 := \wt Jv_0$ and note that
    $\norm{w_0} > 1 - \eps'$.

  \item 
    We now prove eqn.~\eqref{eq:lem.1.7.1}.  Combining~\eqref{eq:1.32}
    and~\eqref{eq:1.26} we see that
    \begin{align*} 
      \norm{\e^{-\im tH} w_0 - J_u \e^{-it A_0} v_0} 
      & \le \norm{\e^{-\im tH} w_0 - \wt J X \e^{-it A_0} v_0} 
        + \norm{(\wt J X - J_u) \e^{-it A_0} v_0} \\
      & \le 2 \eps' + \norm{(1 - \chi_{s,t}) \e^{-it A_0} v_0} \\ 
      & \le 4 \eps' < \frac {4 \eps'} {1 - \eps'} \norm{w_0}
        < \eps \norm{w_0}, 
    \end{align*}
    since $(\wt J X - J_u) \chi_{s,t} \e^{-it A_0} v_0 = 0$ for $t >
    0$, $0 < \eps < 1$, and $\eps' = \eps/5$.

    The proof of~\eqref{eq:lem.1.7.2} is similar and omitted.\qedhere
  \end{myenumerate} 
\end{proof}

It is now easy to prove the main result of this section. 
\begin{theorem}
  \label{thm:1.9}
  The entries of the scattering matrix $(S_{ij})_{i,j \in \{\low,
    \upp\}}$, as defined in Eqn.~\eqref{eq:scattering-matrix}, are all non-zero
  operators.
\end{theorem}
\medskip
\begin{proof}
  \begin{myenumerate}{\roman}
  \item 
    We first show that the operator $S_{\low\upp}$ is non-zero.  Let
    $0 < \eps < 1/4$ and let $v_0$ and $w_0$ be as in
    \Lem{1.7}. Without loss of generality we may assume, in addition,
    that $\norm{w_0} = \norm{v_0} = 1$. Then
    \begin{equation*}
      \scapro{S_{\low\upp} v_0}{v_0} 
      = \scapro{W_-(H,H_{0,\low}, J_\low) v_0}{W_+(H,H_{0,\upp}, J_\upp) v_0}
    \end{equation*} 
    where, by \Lem{1.7},
    \begin{align*}
      \norm{W_-(H,H_{0,\low}, J_\low) v_0 - w_0} < \eps, \qquad
      \norm{W_+(H,H_{0,\upp}, J_\upp) v_0 - w_0} < \eps.
    \end{align*}
    It now follows that $ |\scapro{S_{\low\upp} v_0}{v_0} - 1| \le 3
    \eps < 3/4$. This shows that $S_{\low\upp}$ is non-zero; but
    then, by symmetry, we also have $S_{\upp\low} \ne 0$.

  \item 
    In order to show that $S_{\low\low}$ (and, analogously,
    $S_{\upp\upp}$) is non-zero, it is enough to construct wave
    packets which come in on the lower sheet (limit $t \to -\infty$)
    and which go out on the lower sheet as well (limit $t \to
    +\infty$), up to a small error. It is easy to modify $v_0$ and
    $w_0$ as in \Lem{1.7} to achieve this goal; cf.\ also \Rem{1.10}
    below.  E.g., we may replace the function $\psi_1$ in the proof
    of \Lem{1.7} with $\psi_1(\cdot - k)$ with $\abs k > 1$ so that
    the associated wave packet is located away from the slit at time
    $t = 0$.  We then translate $\Omega$, $\chi$, and $\XX$ in the
    $x$-direction accordingly.  The maps $j$ and $\wt J$ can be simply
    defined as an embedding of $\XX$ into $M_{0,\low}$. We leave the
    details to the reader. \qedhere
  \end{myenumerate}
\end{proof} 

\begin{remark}
  \label{rem:strongly.open}
  In fact, what we obtain here is a particularly strong version of
  openness of the channels in the sense that the norm of the wave
  packet going out on one sheet is close to the norm of the incoming
  state on the other sheet, for suitably chosen states.  For example,
  for any $\eps > 0$ there are states where the norm of the outgoing
  wave packet on the upper sheet is greater than $(1 - \eps)$ times
  the norm of what is coming in on the lower sheet, etc. One might say
  then that the channels are \emph{strongly open}.
\end{remark}

\begin{remark}
  \label{rem:1.10} 
  In dealing with $S_{\low\low}$ we might as well exchange the
  variables $x$ and $y$ and translate in the $y$-direction to avoid
  the slit. In the end, all one needs is a rigid motion of $\XX$ which
  avoids the slit and one gets the impression that ``most'' initial
  states will belong to the range of $S_{\low\low}$ or $S_{\upp\upp}$
  while only a tiny fraction of initial states communicates between
  the two sheets under the evolution $\e^{-\im tH}$. Thus, if one
  wishes to be heard on the upper plane as a member of the lower plane
  one should shout in the right direction (and also rather at a high
  pitch).
\end{remark}
\begin{remark}
Here we give some indications on coverings of the
Euclidean plane with three or more sheets. In the case of three
sheets and two branch points the southern rim of the cut in the sheets numbered
I, II, and III is identified with the northern rim of the sheets numbered
II, III, and I. Then the situation is basically the same as with two sheets
and all channels are open. In the case of four sheets and two
branch points the identification of the rims proceeds as above.
Here we can show that neighboring sheets are open to one another while our
method fails to decide whether the sheets I and III are open
one to another; the same holds for the sheets II and IV. We suspect that
the transmission is very weak (or zero) in the latter cases. 

For three and more sheets there are of course also other possibilities
to connect the sheets along cuts. For three sheets we might look at
two different cuts (and thus four branch points) with sheets I and II
connected along the first cut and sheets II and III connected along
the second cut.  If the two cuts are not aligned we may still
construct wave packets that move from sheet I up to sheet III, up to
small errors. If the two cuts are aligned (i.e., both lie on the real
axis and have positive distance) our method fails. In this last case
we would expect that there is only very weak (or no) transmission from
sheet I to sheet III.

Also note that we are dealing with two (or more) branch points because
a manifold with two sheets and a single branch point---like the
Riemann surface of $\sqrt z$---constitutes just one scattering channel
in our setup.  In this case there is no simple comparison with the
free Laplacian on the Euclidean plane.
\end{remark}


\begin{remark}
  \label{rem:1.3} 
  The singularities at the branch points are only a side issue in our
  investigations. For most of our results, it wouldn't make much of a
  difference if we would ``punch out'' two small holes around the
  branch points and consider the Laplacian with Dirichlet boundary
  conditions on the (smooth) boundaries of these balls.  However, the
  radius of these balls would introduce a parameter which is not well
  motivated and one would have to investigate questions of convergence
  etc.\ as this radius goes to zero.
\end{remark}

For the record, we complement the estimates of \Lem{1.7} with some
further basic properties of $w_0$.
\begin{lemma}
  \label{lem:1.7.2nd.part} 
  Let $P_{\pm, \upp}$ and $P_{\pm, \low}$ denote the projections onto
  the ranges of the wave operators $W_\pm(H, H_{0,\upp}, J_\upp)$ and
  $W_\pm(H, H_{0,\low}, J_\low)$, respectively.  For $\eps > 0$ let
  $w_0$ be as in \Lem{1.7}.  We then have:
  \begin{equation}
    \label{eq:lem.1.7.3a}
   \norm{P_{+,\upp} w_0} > (1 - \eps) \norm{w_0}, \qquad 
     \norm{P_{+,\low} w_0} < \eps \norm{w_0} .
   \end{equation}
   and
   \begin{equation}
     \label{eq:lem.1.7.3b}
  \norm{P_{-,\low} w_0} > (1 - \eps) \norm{w_0}, \qquad 
    \norm{P_{-,\upp} w_0} < \eps \norm{w_0}.     
   \end{equation}
 \end{lemma} 
\begin{proof} 
 We only show~\eqref{eq:lem.1.7.3a}; the proof of \eqref{eq:lem.1.7.3b} is 
 analogous and ommitted. By the Projection Theorem,
    we have
    \begin{align}
      \nonumber
      \norm{P_{\pm, \upp} w_0}
      & = 
      \sup \bigset{\bigabs{\scapro{w_0} \psi}}
      {\psi \in \ran W_\pm(H, H_{0,\upp}, J_\upp),  \;
        \norm \psi = 1} \\
      \label{eq:1.33}
      & =
      \sup \bigset{\bigabs{\scapro{w_0} {W_\pm(H, H_{0,\upp}, J_\upp) \phi}}}
      {\phi \in \HH_\upp, \; \norm \phi = 1},
    \end{align}
    since $\norm{W_\pm(H, H_{0,\upp}, J_\upp)\phi} = \norm{\phi}$ for
    all $\phi \in \HH_\upp$.  In the RHS of eqn.~\eqref{eq:1.33} we
    have
    \begin{equation}
      \label{eq:1.34}
      \scapro[\Lsqr \MM] {w_0} {W_\pm(H, H_{0,\upp}, J_\upp) \phi}
      = \lim_{t \to \pm\infty} 
      \scapro[\Lsqr \MM]{\e^{-\im t H} w_0}{J_\upp \e^{-\im t H_{0,\upp}} \phi}.
    \end{equation}
    In order to obtain a lower bound on $\norm{P_{+, \upp} w_0}$ we
    choose $\phi := v_0$ in Eqn.~\eqref{eq:1.33} and use \Lem{1.7} to
    find
    \begin{equation}
      \label{eq:1.35}
      \Bigabs{\bigscapro[\Lsqr \MM] {\e^{-\im t H} w_0}
                              { J_\upp \e^{-\im t H_{0,\upp}} v_0}
                              - \norm{w_0}^2}   < \eps.  
    \end{equation}
    For an upper bound on $\norm{P_{+,\ell} w_0}$ we use
  $J_\upp J_\ell = 0$, combined with \Lem{1.7}, to see that  
    $\norm{P_{+, \ell} w_0} < \eps$. 
\end{proof}

Of course, one could as well work with the usual formula for the
projection onto the range of a partial isometry. In our case this
formula reads
\begin{equation}
  \label{eq:*}
   P_{\pm, \upp} 
   =  W_\pm(H, H_{0,\upp}, J_\upp) \circ
       \bigl(W_\pm(H, H_{0,\upp}, J_\upp)\bigr)^*. 
\end{equation}


Let us first show that the adjoints $(W_\pm(H, H_{0,\upp}, J_\upp))^*$ of the 
wave operators $W_\pm(H, H_{0,\upp}, J_\upp)$ are given by strong limits,
\begin{equation}
  \label{eq:1.18'}
  \bigl(W_\pm(H, H_{0,\upp}, J_\upp)\bigr)^* 
  =\slimpm  \e^{\im t H_{0,\upp}} J_\upp^* \e^{-\im t H},  
\end{equation}
with $P_\ac(H) = I$.  Since the wave operators $W_\pm(H, H_0, J)$
exist and are complete (and because $J$ satisfies the requirements
of~\cite[p.~36, Prop.~5(c)]{reed-simon-3}, it follows that the wave
operators
\begin{equation}
  \label{eq:1.19'}
  W_\pm(H_0, H, J^*)
  = \slimpm \e^{\im tH_0} J^* \e^{-\im tH} 
\end{equation}
exist. Here $H_0 = H_{0,\low} \oplus H_{0,\upp}$ and $J^* = (P_\low,
P_\upp)$ and we see that
\begin{equation}
  \label{eq:1.20'}
  \e^{\im tH_0} J^* \e^{-\im tH}
  = \bigl( \e^{\im t H_{0,\low}} P_\low \e^{-\im t H}, 
    \e^{\im t H_{0,\upp}} P_\upp \e^{-\im t H} \bigr). 
\end{equation}
The ranges of $ \e^{\im t H_{0,\low}} P_\low \e^{-\im t H} $ and
$\e^{\im t H_{0,\upp}} P_\upp \e^{-\im t H}$ being orthogonal, it is
clear that the strong limit of the left hand side of~\eqref{eq:1.20'}
can only exist if the strong limits of both terms on the right hand
side exist (as $t \to \pm\infty$).

For $w_0$ as in \Lem{1.7} we now compute
\begin{align*}
  \scapro {P_{\pm,\upp} w_0}{w_0}
   = \lim_{t \to \pm\infty}
        \normsqr {\e^{\im t H_{0,\upp}} J_\upp^* \e^{-\im t H} w_0}
   = \lim_{t \to \pm\infty} \normsqr {P_\upp \e^{-\im tH} w_0}
\end{align*}
and the desired result follows by \Lem{1.7}.

%
\section{Perturbations of the Metric}
\label{sec:pert-metric} 
%

We first recall some notions and definitions in Differential Geometry
as used in~\cite{hpw:14}. Given a (smooth) Riemannian metric $g =
(g_{ij})$ on the $\Cispace$-manifold $M$, we denote by $\MM=(M,g)$ the
Riemannian manifold and we let $B_\delta(p)=B_{\delta,\MM}(p)$ denote
the geodesic open ball centered at $p \in M$ with radius $\delta >
0$. For simplicity, we only consider \emph{smooth} metrics $g$ on $M$;
cf., however, the discussion in~\cite{hpw:14} on the non-smooth
case. Our assumptions on $g$ will mainly involve the (sectional or
Gau{\ss}) curvature of $g$ and the injectivity radius. The
\emph{homogenized injectivity radius} $\iota_\MM(p)$ at $p \in M$ is
defined as in~\cite{anderson-cheeger:92} or~\cite[Eqn.~(2.7)]{hpw:14}
by
 \begin{equation}
  \label{eq:2.1}
  \iota_\MM(p) := \sup_{\delta > 0}
     \min \Bigl\{\delta, 
         \inf \bigset{\inj_\MM(y)}{y \in B_{\delta,\MM}(p)} \Bigr\} 
\end{equation}
where $\inj_\MM(y)$ denotes the usual injectivity radius at the point
$y$.  The number $\iota_\MM(p)$ is the largest number $\delta > 0$ for
which the injectivity radius at any $y \in B_\delta(p)$ is not smaller than $\delta$.

The following definition (cf.~\cite[Def.~2.4]{hpw:14}) is of basic importance for our investigations:
\begin{definition}
  \label{def:2.1}
  For a continuous positive function $r_0 \colon M \to (0,1]$ we
  denote by $\Met_{r_0}(M)$ the set of smooth metrics $g$
  on $M$ that satisfy the lower bounds
  \begin{equation}
    \label{eq:2.2}
    \iota_\MM(p) \ge r_0(p), \quadtext{and} 
    \inf \set{\Ric^-_\MM(y)}{y \in B_{r_0(p),\MM}(p)}
    \ge  - \frac 1 {r_0(p)^2}, 
  \end{equation}   
  for all $p \in M$, where $\MM=(M,g)$.
\end{definition}
Since we are in two dimensions, the Ricci curvature equals the Gau{\ss}
curvature (times the metric tensor $g$). The second
condition in eqn.~\eqref{eq:2.2} is a lower bound for the homogenized
Ricci curvature.
Notice that the Euclidean metric $\gEucl = (\delta_{ij})$ on $M$ 
 belongs to $\Met_{r_0}(M)$ if and only if $r_0$ satisfies the condition
\begin{equation}
  \label{eq:2.3}
  r_0(p) 
  \le \frac1 2 \min\{\dist(p,\Qm), \dist(p,\Qp) \}. 
\end{equation}  

We denote by $\Lsqr \MM$ the usual space of (equivalence classes of)
$\Lsqrspace$-integrable functions on the Riemannian manifold
$\MM=(M,g)$ with respect to the Riemannian measure $\dvol_g$.  The
following definition is standard.
\begin{definition}[{cf.~\cite[Def.~3.1]{hpw:14}}]
  \label{def:2.2}
  We say that the Riemannian metrics $g_1$, $g_2$ are
  \emph{quasi-isometric} if there exists a constant $\eta > 0$ such
  that
  \begin{equation}
    \label{eq:2.4}
    \eta g_1(p)(\xi,\xi) \le g_2(p)(\xi,\xi) \le \eta^{-1} g_1(p)(\xi,\xi), 
  \end{equation}
  for all $\xi \in TM$ and $p \in M$. 
\end{definition}
In our case $TM$ can be identified with $\R^2$.  The Hilbert spaces
$\Lsqr{\MM_1}$ and $\Lsqr{\MM_2}$ coincide if $\MM_i = (M, g_i)$ with
$g_1$ quasi-isometric to $g_2$. In this case we let $I$ denote the
natural identification operator mapping a function $f \in
\Lsqr{\MM_1}$ to the same function $f$ in $\Lsqr{\MM_2}$.

We now take a closer look at the property that the metrics $g$ and
$\gEucl$ are quasi-isometric.  Let $A(p)$ be the endomorphism on $TM$
given by $g(p)(\xi,\xi)=\gEucl(A(p)\xi,\xi)$ for all $\xi \in T_pM$
and $p \in M$ and let $\alpha_k(p)$, $k = 1,2$, denote the eigenvalues
of $A(p)$.  If $(g_{ij}(p))$ denotes the matrix representation of $g$
on $T_pM$ in the standard coordinates, then the $\alpha_k(p)$ are
also the eigenvalues of $(g_{ij}(p))$.  Thus $g$ and $\gEucl$ are
quasi-isometric if and only if there is a number $\eta > 0$ such that
$\eta \le \alpha_k(p) \le \eta^{-1}$, for $k = 1, 2$ and for all $p
\in M$.

We are now ready to define the basic distance function $\wt d_1$: Let
$\alpha_1(p)$, $\alpha_2(p)$ denote the eigenvalues of $A(p)$.  We
then define as in~\cite[eqns.~(3.2) and~(3.5)]{hpw:14}
\begin{equation}
  \label{eq:2.5}
  \wt d (\gEucl,g)(p)
  := \max_k \bigabs{\alpha_k(p)^{1/2} - \alpha_k(p)^{-1/2}},
\end{equation}  
\begin{equation}
  \wt d_\infty(\gEucl,g) := \sup_{p \in M} \wt d(\gEucl,g)(p)
  \quadtext{and}
  \label{eq:2.6}
  \wt d_1(\gEucl,g) := \int_M  \wt d(\gEucl,g)(p) r_0(p)^{-4} \dd p. 
\end{equation}
We call $\wt d_1(\gEucl,g)$ the \emph{weighted
  $\Lpspace[1]$-quasi-distance of $g$ and $\gEucl$}; we have dropped
the symmetrizing factor $1 + \rho_{\gEucl,g}(p)$ of $\wt d_1$
appearing in~\cite[eqn.~(3.5)]{hpw:14} (which has no influence on our
estimates because it is a bounded function).

Let us assume now that $g$ is quasi-isometric to the Euclidean metric
$\gEucl$ (this is equivalent with $\wt d_\infty(g,\gEucl)<\infty$),
and denote by $\MM=(M,g)$ the corresponding Riemannian manifold.  Then
there is a (unique) self-adjoint Laplacian $H_g$, acting in the
Hilbert space $\Lsqr \MM$, with quadratic form domain given by the
Sobolev space $\Sobn \MM$, and defined by
\begin{equation}
  \label{eq:2.7}
   \scapro{H_g u} v
   = \int_M \scapro[g]{\nabla u}{\nabla v} \dvol_g
   = \sum_{ij} \int_M g^{ij} \partial_i u \partial_j \conj{v} 
       \sqrt {\det g} \dd p, 
\end{equation}
for any $u \in \dom(H_g) \subset \Sobn \MM$ and $v \in \Sobn \MM$,
where $(g^{ij})$ is the inverse of $(g_{ij})$.  In the Euclidean case
($g = \gEucl$) the operator $H_\gEucl$ agrees with the operator $H$
defined in \Sec{wave-ops-euclid}; recall that $H_\gEucl$ is purely
a.c. From Theorem~3.7 of~\cite{hpw:14} we now obtain the following
result on the existence and completeness of the wave operators.

\begin{theorem}
  \label{thm:2.3}
  Suppose we are given a continuous function $\map{r_0} M {(0,1]}$
  satisfying condition (2.3) and a metric $g \in \Met_{r_0}(M)$ which
  is quasi-isometric to the Euclidean metric $\gEucl$ on $M$.
  We also assume that the difference between $g$ and $\gEucl$
  satisfies the $r_0$-dependent weighted integral condition
  $\wt{d}_1(g,\gEucl) < \infty$ with $\wt{d}_1$ as in~\eqref{eq:2.6}.

  Then the wave operators
  \begin{equation}
    \label{eq:2.8}
    W_\pm(H_g,H_\gEucl,I) =\slimpm \e^{\im tH_g} I \e^{-\im tH_\gEucl} 
  \end{equation} 
  and
  \begin{equation}
    \label{eq:2.8b}
    W_\pm(H_g,H_0,IJ) =\slimpm \e^{\im tH_g} IJ \e^{-\im tH_0} 
  \end{equation} 
  exist and are complete with final subspace $\HH_\ac(H_g)$. 
\end{theorem}
\begin{remark*}
  Under suitable conditions on $g$ the operator $H_g$ will be
  absolutely continuous (cf.\ Donnelly~\cite{donnelly:99},
  Kumura~\cite{kumura:10, kumura:13}. In this
  case the wave operators in \eqref{eq:2.8} are even unitary.
\end{remark*} 
\begin{remark*}
  In applying the fundamental perturbation theorems in~\cite{hpw:14}
  we can deal with the branch points $\Qpm$ in the way described in
  \Rem{1.5}, i.e., we have (formally) a manifold with four ends, with
  two ends given by $M_\ext$ as in Eqn.~\eqref{eq:1.7} and two ends
  given by $B_{1/2}(\Qpm)$.  Again, the ends $B_{1/2}(\Qpm)$ do not
  participate in the scattering.
\end{remark*}

Following the development in Section~5 of~\cite{hpw:14} we next
consider the question of continuity of the scattering matrix and the
openness of the scattering channels for small perturbations of the
Euclidean metric. As in~\cite{hpw:14} we define for $r_0$ as above and
$\gamma, \eps > 0$
\begin{equation*}
  \Met_{r_0}(M, \gEucl, \gamma, \eps) 
  := \bigset{g \in \Met_{r_0}(M)}
         {\wt d_\infty(g,\gEucl) \le \gamma, \quad 
          \wt d_1(g,\gEucl) \le \eps}, 
\end{equation*}
i.e., $ \Met_{r_0}(M, \gEucl, \gamma, \eps) $ is the set of smooth metrics
$g$ on $M$ enjoying the following properties:
\begin{enumerate}
\item
  \label{def.met.i}
  The  homogenized injectivity radius and the homogenized curvature of $g$  
  at $p \in M$ are bounded from below by $r_0(p)$ and by 
  $-1/r_0(p)^2$, respectively. 
\item
  \label{def.met.ii}
  The metric $g$ is quasi-isometric to $\gEucl$ with the bound $\wt
  d_\infty(g,\gEucl) \le \gamma$.

\item
  \label{def.met.iii}
  The weighted $\Lpspace[1]$-quasi-distance $\wt{d}_1(g,\gEucl)$ is
  not larger than $\eps$.
\end{enumerate}
Note that condition~\itemref{def.met.iii} requires a quantitative
smallness of the deviation of $g$ from the Euclidean metric in the
sense that $\wt d_1(g,\gEucl) \le \eps$ while the main assumption in
\Thm{2.3} only stipulates $\wt d_1(g,\gEucl) < \infty$.

Then Theorem~5.1 of~\cite{hpw:14} yields the strong convergence of the
scattering operators as $\eps \downarrow 0$, and Cor.~5.3
of~\cite{hpw:14} establishes the openness of the scattering channels,
for small $\eps > 0$. We are now going to make this precise.

Let $\gamma > 0$ be fixed. For $\eps > 0$, we consider $g_\eps \in
\Met_{r_0}(M, \gEucl, \gamma, \eps)$ and we let $H_{g_\eps}$ denote
the Laplacian of $(M,g_\eps)$. The natural identification operator
from $\Lsqr{M,\gEucl}$ to $\Lsqr{M,g_\eps}$ is written
$I_{g_\eps}$. Then the scattering operator is given by
\begin{equation*}
  S_{g_\eps}
  = S(H_{g_\eps}, H_0, I_{g_\eps}J) 
  = \bigl(W_+(H_{g_\eps}, H_0, I_{g_\eps}J)\bigr)^* 
                \circ W_-(H_{g_\eps}, H_0, I_{g_\eps}J), 
\end{equation*}   
with $H_0$ and $J$ as in \Sec{wave-ops-euclid}, \Prp{1.4}, and the scattering
matrix $(S_{ij})_{i,j \in \{\low,\upp\}}$ is defined by
\begin{equation*}
  S_{ij}(H_{g_\eps}, H_0, I_{g_\eps}J)
  := \bigl(W_+(H_{g_\eps}, H_{0,i}, I_{g_\eps}J_i)\bigr)^* 
                \circ W_-(H_{g_\eps}, H_{0,j}, I_{g_\eps}J_j), 
\end{equation*}
for $i,j \in \{\low, \upp\}$. Then~\cite[Thm.~5.1]{hpw:14} yields the
following result:

\begin{theorem}
  \label{thm:2.4}
  Let $\gEucl$, $H_\gEucl$, $H_0$, $J$ as above, let $S(H_\gEucl,H_0,J)$ as in
  eqn.~\eqref{eq:scattering-operator} , and let $\gamma > 0$.  For
  $\eps > 0$ and $g_\eps \in \Met_{r_0}(M, \gEucl, \gamma, \eps)$ we
  denote by $H_{g_\eps}$ the Laplacian of $\MM_\eps = (M, g_\eps)$ and
  by $I_\eps \colon \Lsqr \MM \to \Lsqr{\MM_\eps}$ the natural
  identification.

  Then there exists $\eps_0 > 0$ such that the scattering operators
  $S(H_{g_\eps}, H_0, I_\eps J)$ converge strongly to $S(H, H_0, J)$,
  as $\eps \to 0$.
\end{theorem}

As in~\cite[Cor.~5.3]{hpw:14}, we immediately obtain a stability
result for the scattering matrix
where we also use the fact, established in \Thm{1.9}, that the
operators $S_{ik}(H, H_0, J)$, $i,k \in \{\low, \upp\}$, are non-zero, 
i.e., all scattering channels are open.  

\begin{corollary}
  \label{cor:2.5}
  For any $\gamma > 0$ fixed, there exists $\eps_0 > 0$ such that
  $S_{ik}(H_{g_\eps}, H_0, I_\eps J) \ne 0$ for all metrics $g_\eps \in
  \Met_{r_0}(M, \gEucl, \gamma, \eps)$ and all $0< \eps \le \eps_0$.
\end{corollary}
%
\section{Examples}
\label{sec:examples}
%

We first illustrate \Thm{2.3} in the special case where the 
perturbed metric on $M$ is associated with the graph of a 
function $f \colon M \to \R$ of class $\Contspace[2]$. As usual, 
we define $\Phi \colon M \to \R^3$ by $\Phi(p) := (p,f(p))$ and  
\begin{equation*}
  g = g_f = J_\Phi^{\mathrm T} \cdot J_\Phi 
  =
  \begin{pmatrix}
    1 + f_x^2 & f_x f_y \\
    f_x f_y & 1 + f_y^2 \\
  \end{pmatrix}, 
\end{equation*}
 where $J_\Phi$ is the Jacobian of $\Phi$.
The eigenvalues of $g$ are $1$ and $\det g = 1 + f_x^2 + f_y^2$. 
The curvature $\kappa$ of $\MM := (M,g)$ is given by the well-known formula
\begin{equation}
  \label{eq:curvature}
  \kappa
  = \frac {f_{xx} \cdot f_{yy} - f_{xy}^2} {(1 + f_x^2 + f_y^2)^2} 
  = \frac {\det H_f} {\det^2 g}   
\end{equation}
(cf.~\cite[p.~163]{docarmo:76}, \cite[eqn.~(14.105)]{gilbarg-trudinger:83}), 
where $H_f$ is the Hessian of $f$. We let 
\begin{equation}
  \label{eq:harm.rad.ex}
  d_0(p) := \min\{1, \dist(p,\Qm), \dist(p,\Qp) \} , \qquad p \in M, 
\end{equation}
where the distances are measured in $(M,\gEucl)$. We have the following
proposition.
\begin{proposition} 
  \label{prp:3.1} 
  Let $\map f M \R$ be of class $\Contspace[2]$ with bounded
  first and second order derivatives and suppose that
  \begin{equation}
    \label{eq:3.3}
    \int_M |\nabla f|^2  d_0^{-4} \dd p < \infty.  
  \end{equation}
  Then the wave operators $W_\pm(H_g, H_\gEucl, J)$ exist and are complete. 
\end{proposition} 
\begin{proof}
  Since $g = g_f$ has the eigenvalues $1$ and $1 + f_x^2 + f_y^2$ with
  $f_x$, $f_y$ bounded, the metric $g$ is quasi-isometric to the
  Euclidean metric on $M$.  We now choose a suitable function $r_0$
  which then defines the class $\Met_{r_0}(M)$.
  Note that the choice of $r_0$ is not unique, and one may 
  obtain different results for different choices. 
  In view of eqn.~\eqref{eq:2.3} the simplest choice appears to be
  $r_0 := \rho\, d_0$ with a constant $\rho \in (0, 1/2]$ which we are
  going to fix now.

  Since $f$ has bounded second order derivatives, the curvature of
  $(M,g_f)$ is bounded in absolute value by some constant $K \ge 0$
  and the second condition in Eqn.~\eqref{eq:2.2} is satisfied
  provided $\rho \le 1/\sqrt K$.  According to \Prp{b.5} there exists
  a constant $c_0 > 0$ such that the (homogenized) injectivity radius
  of $(M,g)$ at $p \in M$ is bounded from below by $c_0 d_0(p)$. We
  may thus pick any $\rho > 0$ satisfying $\rho \le \min\{1/2, 1/\sqrt
  K, c_0\}$.

  In remains to show that $\wt d_1(\gEucl,g)$ as in eqn.~(2.6) is
  finite. Here we first estimate
  \begin{equation*}
    \wt d(\gEucl,g) = \sqrt{1 + f_x^2 + f_y^2} - \frac 1 {\sqrt{1 + f_x^2 + f_y^2}} \le 
    |\nabla f|^2; 
  \end{equation*}
  therefore condition~\eqref{eq:3.3} implies $\wt d_1(\gEucl,g) <
  \infty$. Since the assumptions of \Thm{2.3} are satisfied, we may
  conclude that the wave operators for the pair ($H_\gEucl,H_g)$ exist and
  are complete.
\end{proof}

\begin{remarks}
  \indent
  \begin{myenumerate}{\roman}
  \item Condition~\eqref{eq:3.3} is satisfied if $\nabla f$ is
    square integrable at infinity and decays near $\Qm$ and $\Qp$ like
    \begin{equation*}
      |\nabla f(p)| \le \min\{\dist(p,\Qm), \dist(p,\Qp) \}^{1 + \beta}
    \end{equation*}
    for some $\beta > 0$.

  \item It is illuminating to take a look at other choices of $r_0$
    where $r_0$ tends to zero at infinity.  The class of admissible
    functions $f \colon M \to \R$ that define the perturbed metric
    changes in the following way. On the one hand, the injectivity
    radius associated with the metric $g$ may now go to zero at
    infinity and the (Gau{\ss}) curvature need no longer be bounded
    from below by a constant; on the other hand, it is now more
    difficult to satisfy the weighted integral
    condition~\eqref{eq:3.3}.
  \end{myenumerate}
\end{remarks}

In an analogous way one can indicate simple conditions on $f$ which
allow the application of \Thm{2.4}. We consider functions $f \colon M
\to \R$ of class $\Contspace[2]$ with first and second order
derivatives bounded by some constant $C$ and which are such that $g_f
\in \Met_{r_0}(M)$ with $r_0$ as above.  Then $\wt d_\infty(g_f,
\gEucl) \le C^2$, and we may now choose $\gamma := C^2$.  For $\eps >
0$, the condition $\wt d_1(g_f,\gEucl) \le \eps$ is safisfied if
\begin{equation}
  \label{eq:3.4}
  \int_M |\nabla f|^2  d_0^{-4} \dd p \le \eps;    
 \end{equation}
 in this case, we have $g_f \in \Met_{r_0}(M, \gEucl, \gamma, \eps)$ and
 the results of \Thm{2.4} and \Cor{2.5} apply.
\begin{proposition} 
  Suppose we are given a sequence $(f_n) \subset \Cont[2] M$ enjoying
  the following properties:
  \begin{enumerate}
  \item There is a constant $C \ge 0$ such that $|\partial_j f_n(p)|
    \le C$ and $|\partial_{ij} f_n(p)| \le C$ for all $p \in M$ and
    all $n \in \N$.

  \item We have
    \begin{equation}
      \label{eq:3.5}
      \int_M |\nabla f_n|^2  d_0^{-4} \dd p  \to 0, \qquad n \to \infty. 
    \end{equation}
  \end{enumerate}
  Let $g_n$ denote the metric induced by $f_n$ and let $I_n$ the
  associated natural identification operator, as above.  
  Then the scattering operators $S(H_{g_n}, H_0, I_n J)$ exist and
  converge strongly to $S(H_\gEucl, H_0, J)$, as $n \to \infty$.
\end{proposition} 


%
%

\appendix

%
\section{Self-Adjointness and Spectral Properties.}
\label{app:spectral-properties}  
%


In this appendix we study the Sobolev spaces $\Sobnsymb^1$
and Laplace-Beltrami operators on branched coverings of the
Euclidean plane. Here we are mainly interested in self-adjointness,
compactness properties, and the question of absolute continuity of the
Laplacian.

\subsection{Double covering with a single branch point.}
 
It is convenient to begin the analysis of the Laplacian on branched
coverings with the case of a single branch point, i.e., we look at a
real version of the Riemann surface of $\sqrt z$. In the case of a
single branch point one can use separation of variables in polar
coordinates.  We take the liberty of using the same symbols $M_0$,
$\MM_0$, $H_0$ etc.\ as in the case of two branch points.  For most of
our results the corresponding analogue for the case of two branch
points will be immediate; cf.\ \Sec{a.2}.

Let $M_0$ denote the $\Cispace$-manifold obtained by joining two
copies of $\R^2$ along the line $\{(x,0) \in \R^2 \mid x \le 0\}$ in
the usual crosswise fashion. Equipped with the Euclidean metric tensor
$\gEucl = (\delta_{ij})$ we obtain the Riemannian manifold ${\MM}_0 =
(M_0, \gEucl)$ with the single branch point $(0,0)$.  The origin
$(0,0)$ does not belong to $M_0$ and $\MM_0$ is not complete.
For $r > 0$ we let $B_r \subset M_0$ denote the set of points in
$M_0$ with distance less than $r$ from the origin; the ``discs''
$B_r$ form a two-sheeted covering of the punctured disc $\set{(x,y)
  \in \R^2}{0 < x^2 + y^2 < r^2}$.

In order to define the Laplacian ${H}_0$ of ${\MM}_0$ we consider the
Hilbert space $\HH_0 := \Lsqr {\MM_0}$, with scalar product
denoted by $\scapro \cdot \cdot$, and the Sobolev space $\Sobn
{{\MM}_0}$, given as the completion of $\Cci {M_0}$ with respect to
the norm $\norm[1] \cdot$ defined by
\begin{equation}
  \label{eq:A.1}
    \normsqr[1] \psi
    := \int_{M_0} \abssqr{\psi(x)} + \abssqr{\nabla \psi(x)} \dd x, 
    \qquad \psi \in \Cci {M_0}. 
\end{equation}  
Then ${H}_0$ is defined as the unique self-adjoint operator satisfying
$\dom(H) \subset \Sobn {M_0}$ and
\begin{equation}
  \label{eq:A.2}
  \scapro{{H}_0 u} v 
  = \int_{M_0} \nabla u \cdot \nabla {\conj v} \dd x, 
  \qquad u \in \dom({H}_0), \,\,\, v \in \Sobn {M_0}.
\end{equation}
By elliptic regularity, we have $\dom({H}_0) \subset \Sobloc[2] {M_0}$
and $H_0 u = -\Delta u$ for all $u \in \dom (H_0)$. More precisely, if
$u$ belongs to $\dom(H_0)$, then the restriction of $u$ to $M_0
\setminus \overline{B}_\eps$ belongs to $\Sob[2]{M_0 \setminus
  \overline B_\eps}$, for any $\eps > 0$. We note as an aside that
$\dom({H}_0) \not\subset \Sob[2]{M_0}$. Indeed, the function $\map u
{M_0} \R$, defined in polar coordinates by $u(r, \theta) := \frac 1
{\sqrt r} {\sin r} \cos \frac\theta2 $, satisfies $-\Delta u = u$ in
$M_0$. If we now take any smooth function $\map \phi {M_0} \R$ which
is $1$ on $B_1$ and vanishes outside of $B_2$, say, then $\phi u \in
\dom(H_0)$ but, by a straight-forward calculation, $(\phi u)_{xx}
\notin \Lsqr {M_0}$.

Another natural Sobolev space is the space $\Sob {M_0} = \SobW[2] 1
{M_0}$, consisting of all functions in $\Lsqr {M_0}$ that have first
order distributional derivatives in $\Lsqr {M_0}$.  For an open set
$\Omega \subset \R^d$ with smooth boundary, $\Sobn \Omega$ is
associated with a (weak form of) Dirichlet boundary conditions
while the Laplacian with form domain $\Sob
\Omega$ is called the Neumann Laplacian of $\Omega$.  In the case at
hand, however, the Sobolev spaces $\Sob {M_0}$ and $\Sobn {M_0}$ coincide. 
For completeness, we include the (standard) proof. 

\begin{lemma}
  \label{lem:A.1}
  We have $\Sob {M_0} = \Sobn {M_0}$.
\end{lemma}
\begin{proof}
  Let $0 \le u \in \Sob {M_0}$ and let $u_n := \min\{u, n\}$ for $n
  \in \N$.  Then $u_n \to u$ in $\Sob {M_0}$
  (cf.~\cite{gilbarg-trudinger:83}) and we see that $\Sob {M_0} \cap
  \Linfty {M_0}$ is dense in $\Sob {M_0}$. Consider a sequence of
  Lipschitz continuous functions $\phi_k \colon {M_0} \to [0,1]$ with
  the following properties: $\phi_k$ vanishes on $B_{1/k}$ and
  $\phi_k(x) = 1$ for $x \notin B_{2/k}$; furthermore, there exists a
  constant $c$ such that $|\nabla \phi_k(x)| \le c/k$, for all $k \in
  \N$.  For any $n \in \N$ fixed, we have $\phi_k u_n \to u_n$ in
  $\Lsqr {M_0}$ and $\nabla(\phi_k u_n) \to \nabla u_n$ weakly in
  $(\Lsqr {M_0})^2$, as $k \to \infty$.  Thus, for any $\eps > 0$,
  there exist $n_0 \in \N$ and a (finite) convex combination $v_\eps$
  of the $\phi_k u_{n_0}$ such that $\norm[1]{v_\eps - u_{n_0}} <
  \eps$. But $v_\eps \in \Sobn {M_0}$, and the result follows.
\end{proof}

By \Lem{A.1} 
there is only one self-adjoint extension of the Laplacian on $\Cci
{M_0}$ with form-domain contained in the Sobolev space $\Sob {M_0}$.
On the other hand, it is easy to see that ${H}_0$ is \emph{not}
essentially self-adjoint on $\Cci {M_0}$. Indeed, we may just follow
the line of arguments leading to~\cite[Thm.~X.11]{reed-simon-2}) for
the Laplacian in $\R^2$.  In the present situation, we use separation
of variables in polar coordinates $(r,\theta)$, with $r > 0$ and the
angle variable $\theta$ running through $[0, 4\pi)$ instead of
$[0,2\pi)$. The eigenvalues of the angular operator are now given by
$\kappa_\ell = - \frac1 4 \ell^2$ with $\ell \in \N_0$. As a
consequence, the corresponding radial operators (cf.~eqns.~(X.18)
in~\cite[loc.\ cit.]{reed-simon-2})
 \begin{equation}  
  \label{eq:radial-operators}
    h_\ell :=  -\frac{\dd^2}{\dd r^2} + \frac{\ell^2 - 1}{4 r^2} , \qquad \ell \in \N_0, 
\end{equation} 
 are not essentially self-adjoint on $\Cci{0,\infty}$  for $\ell = 0$ and for $\ell = 1$. 

\medskip

We next consider the Rellich compactness property. In the following
lemma we let $\chi_r$ denote the characteristic function of $B_r$. 
\begin{proposition}
  \label{prp:A.2}
  For all $R > 0$ the operators $ \chi_R ({H}_0 + 1)^{-1/2}$ and
  $({H}_0 + 1)^{-1/2} \chi_R$ are compact.
\end{proposition}
\begin{proof}
  It is clearly enough to show that the mapping $\Sob {M_0} \ni u
  \mapsto \chi_R u \in \Lsqr {M_0}$ is compact.  Away from the origin
  we may apply the standard Rellich Compactness Theorem, but we need a
  different argument in a neighborhood of the origin.
  \begin{myenumerate}{\roman}
  \item 
    \label{lem.1.2.i}
    Let us first show that the embedding $\Sob {M_0} \hookrightarrow
    \Lsqrloc {M_0}$ is compact.  Indeed, any compact subset $K \subset
    {M_0}$ can be covered by a finite number of discs $B_r(\Pt_i)$,
    $i=1,\ldots, n$ with suitable $n \in \N$, $0 < r < \dist\{K,
    (0,0)\}$, and $\Pt_i \in {M_0}$. Then $\clo{B}_r(\Pt_i) \subset
    {M_0}$ and each disc $B_r(\Pt_i)$ is (equivalent to) a Euclidean
    disc in $\R^2$.  We may then use a partition of unity subordinate
    to this covering of $K$ and we may apply the usual Rellich
    Compactness Theorem in each $B_r(\Pt_i)$.

  \item
    \label{lem.1.2.ii}
 Let us define the Dirichlet Laplacian ${H}_{0;1}$ of $B_1$ as the 
 (unique) self-adjoint operator with  quadratic form domain 
 $\Sobn {B_1}$ and with quadratic form~\eqref{eq:A.1}. 
 Using again separation of variables in polar coordinates as above, 
 we have to deal with the Friedrichs extension of the operators  
 $h_\ell$  on $\Cci{0,1}$, for $\ell \in \N_0$. 
 Each of the operators $h_\ell$ has purely discrete spectrum with the 
 lowest eigenvalue tending to $\infty$ as $\ell \to \infty$. 
  It follows that ${H}_{0;1}$ has compact resolvent.   
   
  \item
    \label{lem.1.2.iii}
    Let $(u_k) \subset \Sobn {M_0}$ and suppose that $u_k \to 0$
    weakly in $\Sob {M_0}$.  It is enough to show that $\chi_R u_k \to
    0$ in $\Lsqr {M_0}$ strongly, for all $R > 0$.

    Choose a (smooth) cutoff-function $\phi$ with support in
    $B_1$ and which is equal to $1$ in $B_{1/2}$. We then have
    $\phi u_k \in \Sobn {B_1}$ and $\phi u_k \to 0$
    weakly in $\Sobn{B_1}$. By the second part of this proof 
   $H_{0;1}$ has compact resolvent. This implies that $\phi u_k \to 0$ in
    $\Lsqr{B_1}$ since
    \begin{equation*}
      \norm{\phi u_k}^2 = \scapro{H_{0;1}^{-1}(\nabla(\phi u_k))}
           {\nabla(\phi u_k)}
    \end{equation*}
    where $\nabla(\phi u_k) \to 0$ weakly and
    $H_{0;1}^{-1}(\nabla(\phi u_k)) \to 0$ strongly in
    $\Lsqr{B_{1/2}}$.

    On the other hand, $B_R \setminus B_{1/2}$ is a relatively compact
    subset of ${M_0}$, and therefore $(1 - \phi) \chi_R u_k \to 0$ in
    $\Lsqr {M_0}$ by part~\itemref{lem.1.2.i} of this proof.
  \end{myenumerate}
\end{proof}

We next comment on the spectral properties of ${H}_0$. As ${H}_0 \ge
0$ we have $\spec {{H}_0} \subset [0,\infty)$. Clearly, $\spec[ess]
{{H}_0} \supset [0,\infty)$ and so $\spec {{H}_0} = \spec[ess] {{H}_0}
= [0,\infty)$.  All operators $h_\ell$ in~\eqref{eq:radial-operators}
have purely absolutely continuous spectrum since $0$ is not an
eigenvalue and the operators $h_\ell$ are purely a.c.\ in
$(0,\infty)$; cf., e.g., \cite[Satz~14.25]{weidmann:03}. It is then clear 
that ${H}_0$ is also purely a.c.; in other words, $H_0$ has no
singular continuous spectrum and has no eigenvalues.

\subsection{Double coverings with two branch points.}
\label{sec:a.2}

We now return to the manifold $\MM$ with two branch points and the
associated Laplacian $H$ as in \Sec{wave-ops-euclid}. As in the case of a single
branch point the Sobolev spaces $\Sob \MM$ and $\Sobn \MM$
coincide. Again, $H$ is not essentially self-adjoint on $\Cci M$. Also
$(H + 1)^{-1} \chi_R$ is compact for all $R > 0$ with $\chi_R$ as in
\Sec{wave-ops-euclid}. 
 The proofs require only some obvious modifications as
compared to the case of a single branch point.  As for the spectral
properties of $H$ it is again clear that $\spec {H} = \spec[ess] {H} =
[0,\infty)$ and it remains to deal with the question of absolute
continuity.  Here we refer to some work of
Donnelly~\cite{donnelly:99} and Kumura~\cite{kumura:10, kumura:13} who have
pertinent statements for complete manifolds which are asymptotically
Euclidean. It is clear from their proofs that the presence of a finite
number of branch points can be accomodated.

As an alternative, it is easy to adapt the En{\ss} method of
scattering (cf.\ e.g.\ \cite{reed-simon-3}) to exclude singular continuous spectrum of $H$.
The absence of eigenvalues can be obtained as in the Kato-Agmon-Simon
theorem in \cite{reed-simon-4}:

\begin{proposition} 
  The Laplacian $H$ of $(M,\gEucl)$ has no eigenvalues.
\end{proposition}
\begin{proof}
  Clearly, $0$ cannot be an eigenvalue of $H$ since an eigenfunction
  for the eigenvalue $0$ would have to be constant.  Positive
  eigenvalues can be excluded by following the proof of the
  Kato-Agmon-Simon Theorem~\cite[Thm.~XIII.58]{reed-simon-4} with some
  obvious modifications and simplifications.  In the case at hand, the
  operator $H$ is not essentially self-adjoint on $\Cci M$, but any
  eigenfunction $\psi$ of $H$ is clearly in $\Ci M$ and there is a
  sequence of smooth functions $\psi_n \in \dom(H)$, vanishing outside
  the radius $n$, such that $\psi_n \to \psi$ and $H \psi_n \to H
  \psi$ in $\Lsqr M$, as $n \to \infty$.
\end{proof}  

For the present paper it is quite useful---albeit not essential---to know
that the Laplacian of $(M,\gEucl)$ is purely absolutely continuous. Of
course, it is also natural to ask whether the operators $H_g$ on $M$
with metric $g$ as in \Sec{pert-metric} are purely absolutely
continuous. Here the papers \cite{donnelly:99} and \cite{kumura:10, kumura:13}
mentioned above give sufficient conditions.


%
\section{Stationary Phase Estimates}
\label{app:stationary-phase}  
%

We refer to~\cite{reed-simon-3} for the basics of stationary phase
estimates.  In this appendix we consider two functions $\psi_1$,
$\psi_2 \in \Schwartz \R$ with $\norm{\psi_1} = \norm{\psi_2} = 1$,
and we let $u_0 := \psi_1 \otimes \psi_2 \in \Schwartz{\R^2}$.
We let $h_0$ denote the (unique) self-adjoint extension of $-\frac
{\dd^2}{\dd x^2}$ on $\Cci \R$ and we let $A_0$ denote the (unique)
self-adjoint extension of $-\Delta$ on $\Cci{\R^2}$ so that $A_0
= h_0 \otimes I_{\{y\}} + I_{\{x\}} \otimes h_0$.  We then write
\begin{equation}
  \label{eq:a.1}
   \Psi_i
   = \Psi_i(x,t)
   := \e^{-\im th_0} \psi_i, \qquad i = 1, 2 \quad t \in \R; 
\end{equation}
in particular, we have  
\begin{equation}
  \label{eq:a.2}
  \e^{-\im t A_0} u_0
  = (\e^{-\im th_0} \psi_1) \otimes  (\e^{-\im th_0} \psi_2) 
  = \Psi_1(\cdot,t) \otimes \Psi_2(\cdot,t), 
  \qquad t \in \R.  
\end{equation}
We will be using the following basic estimate on the real line where
$\hat \psi = \FF \psi$ denotes the Fourier transform for $\psi \in
\Schwartz \R$. It is clearly enough to consider $t \ge 0$, in the sequel. 
\begin{lemma}
  \label{lem:a.1}
  Let $\psi \in \Schwartz \R$ with $\norm{\psi} = 1$ and let $\Psi =
  \Psi(\cdot,t) := \e^{-\im t h_0} \psi$.
  \begin{myenumerate}{\roman}
  \item
    \label{lem.a.1.i}
    Suppose $\hat \psi \in \Cci \R$ and let $a \ge 0$ be such that
    $\supp \hat\psi \subset [-a,a]$.  We then have: For any $m \in \N$
    there exists a constant $c_m \ge 0$ such that
    \begin{equation}
      \label{eq:a.3}
      \abs{\Psi(x,t)} \le c_m (1 + |x| + t)^{-m} , 
      \qquad \abs x \ge 2at. 
    \end{equation}

  \item
    \label{lem.a.1.ii}
    Suppose there exists $s \ge 0$ such that $\supp \hat\psi \subset
    [s,s+1]$. We then have: For any $m \in \N$ there exists a
    constant $c_m \ge 0$ such that
    \begin{equation}
      \label{eq:a.4}
      \abs{\Psi(x,t)} \le c_m (1 + |x - 2st| + t )^{-m}, 
      \qquad  x \le 2st, 
    \end{equation}
    for all $m \in \N$, where the constant $c_m$ can be chosen
    independently of $s$.
  \end{myenumerate}
\end{lemma}
\Lem{a.1} is an immediate consequence of classical
stationary phase estimates, as discussed, e.g.,
in Appendix 1 to Section XI.3 of~\cite{reed-simon-3}. 
A motivation for these estimates is that
the ``classically allowed'' region for $\e^{-\im t A_0}(\psi_1 \otimes
\psi_2)$ at time $t \ge 0$ is contained in the rectangle $[-2at,2at] \times
[2st, 2(s+1)t]$ if $\psi_1$ and $\psi_2$ are as in
\Lemenum{a.1}{lem.a.1.i} and~\itemref{lem.a.1.ii}, respectively.

We use the estimates~\eqref{eq:a.1} and~\eqref{eq:a.2} in the
following lemma where
\begin{equation}
  \label{eq:a.5}
  Q_{s,t} 
  := [-st, st] \times [st, \infty)  \subset \R^2,
  \qquad t \ge 0, 
\end{equation}
and $\chi_{s,t}$ is the characteristic function of $Q_{s,t}$.

\begin{lemma}
  \label{lem:a.2}
  Let $u_0 = \psi_1 \otimes \psi_2$ as above where $\psi_1$ satisfies
  the assumptions of \Lem{a.1} $(i)$ and $\psi_2$ satisfies the
  assumptions of \Lemenum{a.1}{lem.a.1.ii}.

  We then have: for any $m \in \N$ there exists a constant $\wt{c}_m
  \ge 0$ such that for $t \ge 0$ and $s \ge 2a$
  \begin{equation}
    \label{eq:a.6}
    \norm{(1 - \chi_{s,t})  \e^{-\im tA_0}u_0}^2 
    \le \wt{c}_m (1 + st)^{1 - 2m}, 
  \end{equation}
  and
  \begin{equation}
    \label{eq:a.7}
    \norm{(1 - \chi_{s,t})  \nabla \e^{-\im tA_0}u_0}^2 
    \le \wt{c}_m (1 + st)^{1 - 2m}.  
  \end{equation}
\end{lemma}
\begin{proof}
  By \Lem{a.1} we have
  \begin{equation*}
    \int_{|x| > st} |\Psi_1(x,t)|^2 \dd x  
    \le 2 c_m^2 \int_{st}^\infty ( 1 + |x|)^{-2m} \dd x 
    = \frac{2 c^2_m} {2m-1} ( 1 + st)^{1-2m}  
  \end{equation*}
  and, similarly,
  \begin{equation*}
    \int_{y < st} \abssqr{\Psi_2(y,t)} \dd y 
    \le  \frac{c^2_m}{2m-1} ( 1 + st)^{1-2m}, 
  \end{equation*}
  for all $t > 0$ and $s \ge 2a$.  Using Eqn.~\eqref{eq:a.2}, we therefore
  obtain
  \begin{align*}
    \normsqr{(1 - \chi_{s,t}) \e^{-\im tA_0} u_0} 
      & \le C_0 \normsqr{\Psi_2(\cdot,t)}
                   \int_{|x| >  st}  \abssqr{\Psi_1(x,t)} \dd x\\
      &  \qquad\qquad {}+ C_0 \normsqr{\Psi_1(\cdot,t)}
                   \int_{y < st} |\Psi_2(y,t)|^2 \dd y\\
      & \le {\wt c}_m (1 + st)^{1-2m}
  \end{align*}
  with (non-negative) constants $C_0$, $\wt{c}_m$ that are independent
  of $s$.  This proves~\eqref{eq:a.6}.  For the
  estimate~\eqref{eq:a.7}, we use the well-known fact that $p := -\im
  \frac{\dd}{\dd x}$ and $\e^{-\im th_0} = \e^{-\im t p^2}$ commute,
  whence
  \begin{equation*}
    \nabla \e^{-\im t A_0} u_0 = (\e^{-\im th_0} \psi_1' \otimes \psi_2, 
    \psi_1 \otimes \e^{-\im th_0} \psi_2').
  \end{equation*} 
  Proceeding as above, we obtain~\eqref{eq:a.7} with a constant depending on
  $\norm{\psi_1'}$ and $\norm{\psi_2'}$.
\end{proof}

We are now ready to provide the basic estimate for the ``localization
error'' as in eqn.~\eqref{eq:1.24}. 

\begin{lemma}
  \label{lem:a.3}
  Suppose $\psi_1$ and $\psi_2$ are as in \Lem{a.2}.  In addition, let
  $\chi \in \Ci{\R^2}$ with $\chi$, $\nabla \chi$, $\Delta\chi$
  bounded, and such that $\supp \chi \subset \bigset{(x,y) \in \R^2}{
    \abs x \le 1 + \abs y}$.  Let $f = f(x,y,t) = - 2\im (\nabla\chi)
  \cdot \nabla \e^{-\im tA_0} u_0 - \im (\Delta \chi) \e^{-\im tA_0}
  u_0 $ with $u_0 = \psi_1 \otimes \psi_2$. We then have: For any
  $\eps > 0$ there exists $s_\eps \ge 0$ such that
  \begin{equation}
    \label{eq:a.8}
    \int_{-\infty}^\infty \norm{f(\cdot,\cdot,t)} \dd t < \eps, 
    \qquad s \ge s_\eps. 
  \end{equation} 
\end{lemma}
\begin{proof} 
  Without restriction we may assume $s \ge 2a$.  We only consider $t >
  0$, the case $t < 0$ being almost identical.

  The set $Q_{s,t} = (-st, st) \times (st, \infty)$ does not
  intersect the double cone $\set{(x,y)}{\abs y < \abs x} \subset
  \R^2$ and $\nabla\chi$ and $\Delta\chi$ vanish on $Q_{s,t}$.  Then
  the estimates~\eqref{eq:a.6} and~\eqref{eq:a.7} immediately imply
  that for any $m \in \N$ there exists a constant $C_m \ge 0$ such
  that
  \begin{equation*}
    \norm{f(\cdot,\cdot,t)} \le C_m (1 + st)^{(1-2m)/2},
    \qquad s \ge 2a. 
  \end{equation*} 
  We now fix some $m \ge 2$ and integrate with respect to $t$ to
  obtain
  \begin{equation*}
    \int_0^\infty \norm{f(\cdot,\cdot,t)} \dd t
    \le 
    C_m \int_0^\infty  ( 1 + ts)^{(1-2m)/2} \dd t < \infty, 
  \end{equation*} 
  where the integral on the right hand side tends to zero, as $s \to
  \infty$.
\end{proof}

%
\section{Lower Bounds for the Injectivity Radius}
\label{app:lower-bound-inj-rad}  
%

Lower bounds for the injectivity radius are crucial for the
applicability of our results to concrete examples. We are now going to
explain how a comparison result of M\"uller and
Salomonsen~\cite{mueller-salomonsen:07} can be used to deal with
various situations where the metric is associated with the graph of a
function on $\R^2$ or on $\R^2 \setminus \{(0,0)\}$.  These estimates
may be of independent interest. Appendix~D of~\cite{hpw:14} contains
related results for radially symmetric manifolds. Let us first recall
the basic comparison result:

\begin{proposition}[{\cite[Prop.~2.1]{mueller-salomonsen:07},%
    ~\cite[Prop.~D.1]{hpw:14}}]
  \label{prp:b.1}
  Let $M$ denote a smooth $n$-dimensional manifold.  Suppose that the
  Riemannian manifolds $\MM_0 := (M,g_0)$ and $\MM_1 : = (M,g_1)$
  are complete with quasi-isometric metrics $g_0$ and $g_1$, i.e.,
  \begin{equation*}
    \eta g_0 \le g_1 \le \eta^{-1} g_0,
  \end{equation*} 
  for some constant $\eta \in (0,1]$; cf.~\Def{2.2}.  Furthermore,
  suppose that the sectional curvature of $\MM_0$ and $\MM_1$ is
  bounded (in absolute value) by some constant $K \ge 0$. Let
  $\inj_{\MM_0}(p)$ and $\inj_{\MM_1}(p)$ denote the
  injectivity radius of $\MM_0$ and $\MM_1$, respectively, at the
  point $p \in M$.  We then have
  \begin{equation}
    \label{eq:inj.rad.est.1}
    \inj_{\MM_1}(p)
    \ge \frac12 \min 
           \Bigl\{ \frac {\eta^2 \pi} {\sqrt K}, 
                   \eta \inj_{\MM_0}(p) \Bigr\},
    \qquad  p \in M.
  \end{equation} 
\end{proposition}

Note that the assumptions of \Prp{b.1} are global and that the
manifolds are assumed to be complete. We will use simple cut-offs and
also an extension procedure for functions of class $\Contspace[2]$ to
obtain local versions.

In the sequel, we will deal with the special case $n=2$, $\MM_0 =
(\R^2, \gEucl)$ and $\MM_1 = (\R^2, g_f)$ where the metric $g_f$ comes from
a function $f \colon \R^2 \to \R$ of class $\Contspace[2]$, as in
\Sec{examples}.  We start with the particularly simple case where the
first and second order derivatives of $f$ are bounded.

\begin{proposition}
  \label{prp:b.2}
  Let $f \in \Contspace[2](\R^2, \R)$ with bounded derivatives of the
  first and second order. Let $g = g_f$ as defined above, and let
  $\MM_1 := (\R^2, g_f)$. If $ \beta \ge 0$ and $\gamma > 0$ are constants such
  that
  \begin{equation*}
    |D_i f(p)| \le \beta, \qquad |D_{ij}(p)| \le \gamma, 
  \end{equation*}
  for all $p \in \R^2$ and $i,j \in \{1,2\}$, then the radius of
  injectivity of $\MM_1$ at $p \in \R^2$ satisfies
  \begin{equation}
    \label{eq:inj.rad.est.2}
    \inj_{\MM_1}(p) 
    \ge \frac \pi {2 \sqrt 2} \cdot \frac 1 {(1 + 2\beta^2)^2 \gamma}, 
    \qquad p \in \R^2. 
  \end{equation}
\end{proposition}
\begin{proof}
  The injectivity radius of $\MM_0 = (\R^2, \gEucl)$ is infinite at all $p
  \in \R^2$.  As for the constants $\eta$ and $K$ in \Prp{b.1}
  we may take $\eta := 1/(1 + 2\beta^2)$ and $K := 2 \gamma^2$ since
  the curvature $\kappa$ satisfies $|\kappa(p)| \le |\det H_f(p)|$ by
  eqn.~\eqref{eq:curvature}. The desired estimate now follows from
  eqn.~\eqref{eq:inj.rad.est.1}.
\end{proof} 

Henceforth we will drop the factor $\pi/(2\sqrt 2) > 1$ for better
readability.  We next consider $f \in \Contspace[2](\R^2, \R)$ without
assuming a bound for the derivatives of $f$.

\begin{proposition}
  \label{prp:b.3}
  Let $M = \R^2$ and let $f \in \Contspace[2](\R^2, \R)$.  Let $g =
  g_f$ as defined above, and let $\MM_1 := (\R^2, g_f)$.  For $p_0 \in
  M$, let
  \begin{align}
    \label{eq:beta.pzero.1}
    \beta(p_0) 
    &:= \max \bigset{|D_i f(p)|}{|p - p_0| \le 2, \> i \in \{ 1,2\} },\\
    \gamma(p_0)
    &:= \max \bigset{|D_{ij} f(p)|}{|p - p_0| \le 2, \> i,j \in \{1,2\} }.  
  \end{align}
  Then there is a constant $c \ge 0$, which is independent of $f$, such
  that the radius of injectivity of $\MM_1$ at $p_0 \in \R^2$ satisfies
  \begin{equation}
    \label{eq:inj.rad.est.3}
    \inj_{\MM_1}(p_0) 
    \ge \min \bigl\{ 1, (1 + 2c^2\beta(p_0)^2)^{-2}  
    (\gamma(p_0) + c\beta(p_0))^{-1} \bigr\}. 
  \end{equation}
\end{proposition}
\begin{proof}
  Let $p_0 \in \R^2$. We may assume $f(p_0) = 0$
  without restriction of generality since the lower bound
  of~\eqref{eq:inj.rad.est.1} depends only on derivatives of $f$.  Let
  $\phi \in \Cci{B_2(p_0)}$ satisfy $0 \le \phi \le 1$ and $\phi(p) =
  1$ for $p \in B_1(p_0)$.  The function $\tilde f := \phi f$ has
  support contained in $B_2(p_0)$.  Since $f(p_0) = 0$ we have
  $|\tilde f(p)| \le 2 \sqrt 2 \beta(p_0)$ for all $p \in B_2(p_0)$ by
  the mean value theorem. Routine calculations then lead to the
  estimates
  \begin{align*}
    |D_i \tilde f(p)|  \le c_\phi \beta(p_0), \qquad 
    |D_{ij} \tilde f(p)|  \le \gamma(p_0) + c_\phi \beta(p_0), 
  \end{align*}
  for all $p \in \R^2$ and $i,j \in \{1, 2\}$, where $c_\phi$ is a
  constant depending only on a bound for the first and second order
  derivatives of $\phi$. We may also assume that these bounds are
  independent of $p_0 \in \R^2$.  Applying the
  estimate~\eqref{eq:inj.rad.est.2} with $c_\phi \beta(p_0)$ and
  $\gamma(p_0) + c_\phi\beta(p_0)$ replacing $\beta$ and $\gamma$,
  respectively, we obtain the desired result.
\end{proof}

In order to deal with branch points or other singularities, we now
consider functions $f$ on the punctured plane $\R^2_* := \R^2
\setminus \{(0,0)\}$.  The method used in the proof of \Prp{b.3} could
be easily adapted to the case where $p_0$ is close to the
origin. However, this would require working with cut-offs $\phi$ which
are supported in $B_{2 \rho}(p_0)$ and which are equal to $1$ on
$B_\rho(p_0)$, for some $0 < \rho < \frac 1 2 |p_0|$. In this case the
constant $c_\phi$ in the proof of the
estimate~\eqref{eq:inj.rad.est.3} would blow up like $|p_0|^{-2}$, as
$p_0 \to (0,0)$.  We therefore first restrict $f$ to a suitable
half-disc (with positive distance to the origin) and then use a
$\Contspace[2]$-extension method.

\begin{proposition}
  \label{prp:b.4}
  For $f \in \Contspace[2](\R^2_*, \R)$ we define the metric $g = g_f$
  on $\R^2_*$ as before and we let $\MM_1 := (\R^2_*, g_f)$.  For $p_0
  \in \R^2_*$ we consider the annulus
  \begin{align*}
    A(p_0) := 
    \{ p \in \R^2 \mid \frac 1 2 |p_0| \le  |p| \le  \frac 1 2 |p_0| + 2\}
  \end{align*} 
  and we define 
  \begin{align*}
    \beta(p_0) &  
    :=  \max \bigset{|D_i f(p)|}{p \in A(p_0),  \medskip  i \in \{1,2\}},\\
    \gamma(p_0) & 
    :=  \max \bigset{|D_{ij} f(p)|}{p \in A(p_0), \medskip i,j \in \{1,2\}}. 
  \end{align*}
  Then the radius of injectivity of $\MM_1$ at $p_0 \in \R^2_*$ with
  $|p_0| \le 1$ satisfies the lower bound
  \begin{equation}
    \label{eq:inj.rad.est.4}
    \inj_{\MM_1}(p_0) 
    \ge \min \bigl\{ |p_0|/2, (1 + 2c^2{\beta}(p_0)^2)^{-2} 
           ({\gamma}(p_0) + c{\beta}(p_0))^{-1} \bigr\}, 
  \end{equation}
  where $c \ge 0$ is a constant which can be chosen to be independent
  of $f$ and $p_0$.
\end{proposition} 
\begin{proof}
  Without restriction of generality we may assume that $p_0 = (x_0,
  0)$ with $0 < x_0 \le 1$. For the following construction we refer to
  \Fig{4}. We write $r_0 := \frac 1 2 x_0$ and we let $p_1 := (\frac
  1 2 x_0, 0) = \frac 1 2 p_0$. Then the circle $\partial
  B_{r_0}(p_0)$ passes through the point $p_1$. Let $H_+(p_1) \subset
  \R^2$ denote the half-plane to the right of $p_1$, i.e.,
  \begin{equation*}
    H_+(p_1) := \{ p = (x,y) \in \R^2 \mid 2 x >  x_0 \}. 
  \end{equation*}
  It is easy to see that $B_{r_0}(p_0)$ is contained in $B_{1,+}(p_1)
  := H_+(p_1) \cap B_1(p_1)$. Furthermore, $B_{2,+}(p_1) := H_+(p_1)
  \cap B_2(p_1)$ is contained in the annulus $A(p_0)$.
  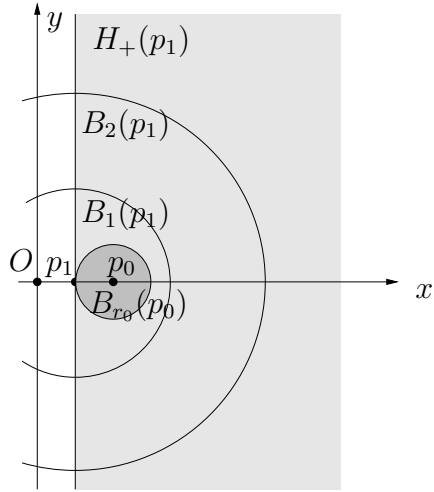
\begin{figure}[h]
    \centering
    \begin{picture}(0,0)%
      \includegraphics[scale=0.5]{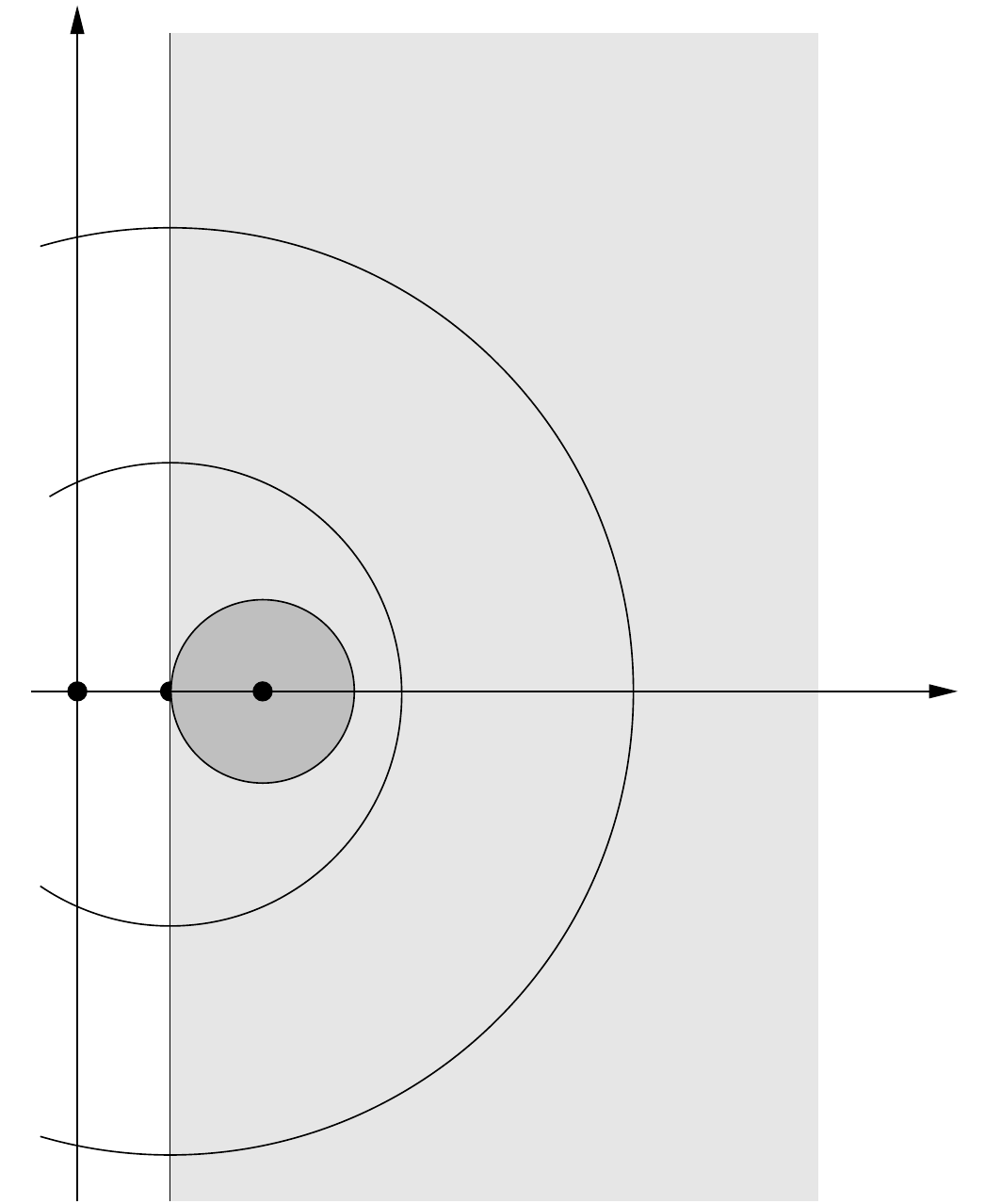}%
    \end{picture}%
    \newlength{\unitlengthorg}
    \setlength{\unitlengthorg}{4144sp}%
    \setlength{\unitlength}{0.5\unitlengthorg}%
\begin{picture}(4845,5829)(176,-6048)
  \put(4906,-3751){$x$}%
  \put( 091,-3426){$O$}%
  \put(531,-3426){$p_1$}%
  \put(1081,-781){$H_+(p_1)$}%
  \put(1261,-3426){$p_0$}%
  \put(946,-1781){$B_2(p_1)$}%
  \put(1036,-3931){$B_{r_0}(p_0)$}%
  \put(946,-2851){$B_1(p_1)$}%
  \put(541,-511){$y$}%
\end{picture}%
    \caption{The reflection method.}
    \label{fig:4}
  \end{figure}

  We now apply the well-known formula for the extension of a function
  of class $\Contspace[2]$ across a hyperplane as in~\cite[Lemma
  6.37]{gilbarg-trudinger:83} to obtain an extension $F$ of $f$ from
  the (closure of) the half-disc $B_{2,+}(p_1)$ into the disc
  $B_2(p_1)$ satisfying the following estimates, valid for all $p \in
  B_2(p_1)$:
 \begin{align*}
   |D_i F(p)| & \le 
        C_\ext \max\bigset{|D_i f(q)|}{q \in \clo B_{2,+}(p_1)} 
  \le C_\ext \beta(p_0),      \\
    |D_{ij} F(p)| & \le 
        C_\ext \max\bigset{|D_{ij} f(q)|}{q \in \clo B_{2,+}(p_1)} 
   \le C_\ext \gamma(p_0), 
 \end{align*} 
 for some constant $C_\ext \ge 0$ as in~\cite[loc.\
 cit.]{gilbarg-trudinger:83}.
 We may now proceed as in the proof of \Prp{b.3}: choose a cut-off
 function $\phi \in \Cci {B_2(p_1)}$ satisfying $\phi(p) = 1$ for all
 $p \in B_1(p_1)$ and let $\tilde f := \phi F$.  We then take $c :=
 c_\phi C_\ext$ with $c_\phi$ as in the proof of \Prp{b.3},
 and the desired estimate follows as before.
\end{proof} 

\begin{remark*}
  Higher order reflections are just one method of obtaining extensions
  of functions of class $\Contspace[2]$. In the case of \Prp{b.4} the
  geometry is particularly simple and we can use reflection at a
  line. Here the orders of differentiation are not mixed in the sense
  that the bounds for the $k$-th order derivatives of the extended
  function $F$ depend solely on bounds for the $k$-th order
  derivatives of $f$, for $k = 1,2$.

  In a more complicated geometric setting, one could work with
  extension from the closed disc $\overline{B}_{r_0}(p_0)$
  using~\cite[Lemma 6.37]{gilbarg-trudinger:83} or employing an
  extension theorem of Whitney type as
  in~\cite[Sec.~VI.2.3]{stein:70}. An advantage of Whitney extension
  lies in the fact that the constant $C_\ext$ can be chosen to be
  independent of the size of the disc $\overline{B}_{r_0}(p_0)$; on
  the other hand, Whitney extension would involve bounds on some
  H\"older-norm for the second order derivatives.
\end{remark*}

We finally return to $M$ as in the body of the paper, with the branch
points $\Qpm$. This is the case which is needed in Section 3. We have
the following result. 
\begin{proposition}
  \label{prp:b.5}
  Let $M$ be the double covering of $\R^2$ with the branch points $\Qpm$ 
 and let $f \in \Contspace[2](M, \R)$ with bounded first and
  second order derivatives.  We define the metric $g = g_f$ on $M$ as
  before and we let $\MM_1 := (M, g_f)$.  Then there is a constant
  $c_f > 0$ such that the radius of injectivity of $\MM_1$ at $p \in
  M$ satisfies the lower bound
  \begin{equation}
    \label{eq:inj.rad.est.5}
    \inj_{\MM_1}(p) 
    \ge c_f \min \{1, \dist(p, q_+), \dist(p,q_-)\}. 
  \end{equation}
\end{proposition}  
\begin{proof}
  If $p_0 \in M$ has distance at least $2$ to $\Qpm$, the
  estimate~\eqref{eq:inj.rad.est.3} applies. In the other cases we may
  proceed as in the proof of \Prp{b.4} with some more or less obvious
  modifications which we indicate now:
  %
  \begin{enumerate}
  \item Since the distance between $\Qpm$ is $2$, we need to scale
    down all sizes in the proof of \Prp{b.4} by a factor smaller than
    $1$.

  \item The annulus $A(p_0)$ will now run through both sheets.
    
  \item Since the first and second order derivatives of $f$ are
    bounded, the numbers $\beta(p_0)$ and $\gamma(p_0)$ can be
    estimated uniformly by a fixed constant.
  \end{enumerate}
\end{proof}

%

\providecommand{\bysame}{\leavevmode\hbox to3em{\hrulefill}\thinspace}

\end{document}